\documentclass{article}
\usepackage{graphicx}
\usepackage[utf8]{inputenc}
\usepackage{indentfirst,csquotes}

\usepackage[bottom=1.5in]{geometry}
\usepackage[english]{babel}
\usepackage{amsthm}
\usepackage{tikz}
\usetikzlibrary{patterns, decorations.pathmorphing}
\usetikzlibrary{decorations.pathreplacing,calligraphy}

\usepackage{fancyhdr} 
\usepackage{nicematrix}

\usepackage{xcolor}

\makeatletter
\newenvironment{customproof}[1]{%
  \par\pushQED{\qed}%
  \normalfont\topsep6pt \trivlist
  \item[\hskip\labelsep\itshape
    Proof of #1\@addpunct{.}]\ignorespaces
}{%
  \popQED\endtrivlist\@endpefalse
}
\makeatother

\usepackage{cite} 
\usepackage{hyperref}
\usepackage{amssymb}
\usepackage{mathtools}
\usepackage{amsmath}
\usepackage{latexsym}
\usepackage{esint}
\usepackage{wasysym}
\usepackage{pst-eucl}
\usepackage{comment}
\usepackage{hyperref}
\usepackage{blkarray}
\usepackage{soul}
\hypersetup{
    colorlinks=true,
    linkcolor=blue,
    citecolor=teal
    }

\let\P\undefined 
\newcommand{\norm}[1]{\left\lVert#1\right\rVert}
\newcommand{\nor}[1]{\left\lvert#1\right\rvert}

\newtheorem{Theorem}{Theorem}
\numberwithin{Theorem}{section}
\newtheorem{Proposition}[Theorem]{Proposition}
\newtheorem{Lemma}[Theorem]{Lemma}
\newtheorem{Corollary}[Theorem]{Corollary}
\theoremstyle{definition}
\newtheorem{Definition}[Theorem]{Definition}
\theoremstyle{remark}
\newtheorem{Remark}[Theorem]{Remark}
\numberwithin{equation}{section}

\newcommand{\Cza}{{C^{0,\alpha}}}
\newcommand{\Coa}{{C^{1,\alpha}}}
\newcommand{\Cta}{{C^{2,\alpha}}}

 \DeclareMathOperator{\C}{\mathcal{C}}
 \DeclareMathOperator{\R}{\mathbb{R}}

\DeclareMathOperator*{\tr}{tr}
\DeclareMathOperator*{\osc}{osc}
\newcommand*{\Mp}{\mathcal{M}^{+}}
\newcommand{\Mm}{\mathcal{M}^{-}}
\newcommand{\p}{\partial}

\newcommand{\A}{\mathcal{A}}
\newcommand{\P}{\mathcal{P}}
\newcommand{\uu}{\underline{u}}
\newcommand{\bu}{\overline{u}}
\newcommand{\vep}{\varepsilon}

\newcommand{\dini}{\operatorname{Dini}}
\newcommand{\usc}{\operatorname{USC}}
\newcommand{\lsc}{\operatorname{LSC}}
\usepackage{dirtytalk}

\begin{document}
\title{Fully nonlinear parabolic fixed transmission problems}

\author{
David Jesus
\and
Mar\'ia Soria-Carro
}
\newcommand{\Addresses}{{% additional braces for segregating \footnotesize
  \bigskip
  \footnotesize

    David Jesus, \textsc{Dipartimento di Matematica,
Universit\`a di Bologna\\ 
Piazza di Porta San Donato 5, 40126 Bologna, Italy}\par\nopagebreak
  \textit{E-mail address:} {\tt david.jesus2@unibo.it}\\
  \medskip
  
    Mar\'ia Soria-Carro, \textsc{Department of Mathematics, Rutgers University\\ 
110 Frelinghuysen Rd., Piscataway, NJ
08854, USA}\par\nopagebreak
  \textit{E-mail address:} {\tt maria.soriacarro@rutgers.edu}\\
 }

}

\date{\today}

\maketitle
{\bf Abstract.} 
 We consider transmission problems for parabolic equations governed by distinct fully nonlinear operators on each side of a time-dependent interface. 
 We prove that if the interface is $C^{1,\alpha}$, in the parabolic sense, then viscosity solutions are 
 piecewise $C^{1,\alpha}$ up to the interface. As byproducts, we obtain a new ABP-Krylov-Tso estimate, and establish existence, uniqueness, a comparison principle, and regularity results for the flat interface problem. 
\thispagestyle{fancy}
\fancyhead{} 
\fancyfoot{}
\fancyfoot[L]{\footnoterule {\small  
\textit{2020 Mathematics Subject Classification.} Primary: 35B65, 35J60. Secondary: 35K10, 35D40. \\
The first author is partially supported by PRIN 2022 7HX33Z - CUP J53D23003610006. The second author is partially supported by NSF grant DMS-2247096.
}}

\section{Introduction}\label{Sec:Intro}

In this paper, we study the regularity of  continuous viscosity solutions to the following transmission problem for fully nonlinear parabolic equations:

\begin{equation} \label{eq:main1}
\begin{cases}
        \p_t u - F^{\pm}(D^2u)= f^{\pm} & \text{ in } \Omega^\pm, \\
        u_\nu^+-u_\nu^-=g & \text{ on } \Gamma,
\end{cases}
\end{equation}
where the domains $\Omega^\pm$ and the interface $\Gamma$ are defined as
\begin{align*}
\Omega^\pm &:=\{ (x,t) \in\C_1 : \pm (x_n-\psi(x',t))> 0\},\\
\Gamma &:= \{(x,t) \in \C_1 :  x_n=\psi(x',t)\},
    \end{align*}
for some given functions $\psi : \R^{n-1}\times \R \to \R$, $f^\pm: \Omega^\pm \to \R$, and $g : \Gamma \to \R$.  Here, $\nu$ is the spatial unit normal vector to $\Gamma$ pointing towards $\Omega^+$, and  $u_\nu^\pm$ are the normal derivatives of $u^\pm:=u|_{\overline {\Omega^\pm}}$. We denote the backwards unit cylinder by  $\C_1:=B_1\times (-1,0]$.  

Let $\mathcal{S}(n)$ be the set of $n\times n$ symmetric matrices. We further assume that $F^\pm:\mathcal{S}(n)\to\mathbb{R}$ are uniformly $(\lambda,\Lambda)$-elliptic and $F^\pm(0)=0$, i.e., there exist $0<\lambda\leq \Lambda$ such that
\[
	\lambda |N|\leq F^\pm(M+N)-F^\pm(M) \leq \Lambda|N|,
\]
for every $M,N\in\mathcal{S}(n)$, with $N\geq 0$, where $|N|$ denotes the norm of $N$.

Transmission problems model physical phenomena in which the behavior changes across some fixed interface, and have attracted considerable attention throughout the years, starting with the pioneering work of Picone \cite{picone1995} in elasticity in the 1950s and subsequent works \cite{Campanato1959,Lions,Schechter,Stampacchia}. For a comprehensive study of these problems, see \cite{Borsuk}. 
For other recent developments in the elliptic setting, see \cite{CSCS,Citti-Ferrari,DFS2018,Giovagnoli-Jesus2024,Kriventsov,Li-Vogelius,soria2023regularity} and references therein.

Transmission problems naturally arise in the analysis of free boundary problems, which have long been a central topic in the theory of partial differential equations. A foundational approach to such problems was developed by Caffarelli in his seminal works \cite{Caffarelli1987,Caffarelli1989}, where techniques based on monotonicity formulas and Harnack inequalities became standard tools in the field. In a breakthrough contribution, De Silva \cite{DeSilva2011} introduced a simpler and more versatile method to show that flat free boundaries are of class $\Coa$. Her approach, which allows to consider variable coefficients and a source term, is based on deriving a Harnack-type inequality that can be iterated to trap the free boundary between two arbitrarily close planes. In the two-phase setting, this iterative scheme leads, in the limit, to the study of a transmission problem across a fixed interface. 

A particularly important parabolic free boundary problem is the Stefan problem, which was studied by Athanasopoulos, Caffarelli, and Salsa in the seminal works \cite{ACS-Acta,ACS-Annals,ACS-CPAM}. In this series of papers, they prove that flat free boundaries of solutions to the homogeneous two-phase Stefan problem are $\Coa$. More recently, De Silva, Forcillo, and Savin \cite{DeSilva-Forcillo-Savin2021} developed a new approach, which follows more closely the elliptic strategy in \cite{DeSilva2011}, to prove a similar regularity result for the homogeneous one-phase Stefan problem;
see also \cite{Ferrari-Forcillo-Giovagnoli-Jesus2024,Ferrari-Giovagnoli-Jesus2025} for the inhomogeneous one-phase problem. Adopting this strategy to the two-phase Stefan problem would boil down to studying a parabolic fixed transmission problem, albeit with a transmission condition involving the time derivative of the solution. While these problems lie beyond the scope of the present work, the techniques developed here may offer useful insights for their analysis.

 The monograph \cite{LSU1968} provides a rigorous analytical framework and tools that are useful for studying parabolic equations in divergence form with discontinuous coefficients, which arise in transmission problems.
Regarding the viscosity formulation, very little has been done. Recently, the second author, together with Kriventsov, consider in \cite{Kriventsov-Soria-Carro-2024} a \textit{flat} parabolic transmission problem as a tool to study the regularity of the free boundary in a two-phase problem. 

The goal of this paper is to extend the results of the second author and Stinga  in \cite{soria2023regularity} to the parabolic setting. In addition to this extension, we also refine and improve upon several aspects of the original work. First, we remove the assumption of proximity between the operators $F^\pm$, i.e.,
\begin{equation*} 
\sup_{M\in \mathcal{S}(n) \setminus \{0\}} \frac{|F^+(M)-F^-(M)|}{|M|} \leq \theta \ll 1,
\end{equation*}
which was required in the case $g\approx 0$ (see \cite[Lemma~5.2]{soria2023regularity}). We show that the stability argument presented in \cite[Lemma~5.6]{soria2023regularity} for $g \approx 1$  remains valid when $g\approx 0$, thereby eliminating the need for the proximity condition.
 Second, we provide a  simplified proof of the stability argument by introducing a more direct and streamlined approach, avoiding the \textit{closedness} result (see \cite[Lemma~5.1]{soria2023regularity}). 

One of the main tools developed in this paper is a novel ABP-Krylov-Tso estimate (see Theorem~\ref{Thm:ABP}) for supersolutions to \eqref{eq:main1}. Its proof differs from the elliptic case mainly because of the different character of the parabolic convex envelope.  Furthermore, we use a new Hopf lemma for $C^{1,\alpha}$ boundaries, which might be of independent interest. We prove this lemma in the appendix for more general domains satisfying an interior $C^{1,\mathrm{Dini}}$ condition (see Theorem~\ref{thm:hopf}).

Another essential tool is the Harnack inequality (see Lemma~\ref{Lm:Harnack_new}), which presents additional challenges in our setting.
In the elliptic case, a standard strategy involves applying the interior Harnack inequality in a small ball contained on one side of the interface, and constructing a barrier that satisfies the transmission condition. By comparing this barrier with the solution, one can transfer the information from the initial ball across the interface. In the parabolic case, a similar approach can be employed. However, because of the waiting time in the parabolic Harnack inequality, the construction of a suitable barrier in this geometry becomes more difficult.

Our main results are the following; see Sections~\ref{sec:notation}-\ref{sec:viscosity} for the relevant notation. 

\begin{Theorem}[$\Cza$ regularity] \label{thm:intholder}
Let $f \in C(\C_1\setminus \Gamma) \cap L^{\infty}(\C_1)$, $g \in L^{\infty}(\Gamma)$, and $\psi \in C^{1,\alpha}(\overline{\C_1'})$, for some $0 < \alpha < 1$. 
Let $u$ satisfy
\[
\begin{cases}
u \in \mathcal{S}^*(f^{\pm}) & \text{in } \Omega^{\pm}, \\
u_{\nu}^+ - u_{\nu}^- = g & \text{on } \Gamma,
\end{cases}
\]
in the viscosity sense.
Then $u \in C^{0,\alpha_1}(\overline\C_{1/2})$, and
\[
\|u\|_{C^{0,\alpha_1}(\overline \C_{1/2})} \leq C \left( \|u\|_{L^{\infty}(\C_1)} + \|g\|_{L^{\infty}(\Gamma)} + \|f\|_{L^{n+1}(\C_1)} \right),
\]
where $0 < \alpha_1 < 1$ and $C > 0$ depend only on $n$, $\lambda$, $\Lambda$, $\alpha$, and $\|\psi\|_{C^{1,\alpha}(\overline{\C_1'})}$.
\end{Theorem}

 Let $0<\bar \alpha <1$, depending only on $n$, $\lambda$, and $\Lambda$, be the exponent from the $C^{1,\bar \alpha}$ regularity of viscosity solutions to the homogeneous flat interface problem given in Proposition~\ref{prop:regflat}.

\begin{Theorem}[$\Coa$ regularity] \label{thm:globalest}
Let $0<\alpha<\bar\alpha$. Assume that $\psi \in C^{1,\alpha}(\overline{\C_1'})$, $\psi\not\equiv 0$, $g\in C^{0,\alpha}(\Gamma)$, and $f^\pm$ satisfy that there are constants $C_{f^\pm}>0$ such that, for any $(x_0,t_0) \in \C_1$,
$$
\left( \fint_{\C_r(x_0,t_0)\cap\Omega^{\pm}} |f^{\pm}|^{n+1} \, dxdt \right)^\frac{1}{n+1} \leq C_{f^\pm} r^{\alpha-1}, \quad \text{for all } r > 0 \text{ small}.
$$
 If $u$ is a bounded viscosity solution to \eqref{eq:main1}, then $u^\pm \in C^{1,\alpha}(\overline \Omega_{1/2}^\pm)$, with
$$
\|u^\pm\|_{C^{1,\alpha}(\overline \Omega_{1/2}^\pm)} \leq C \|\psi\|_{C^{1,\alpha}(\overline{\C_1'})} \big(\|u\|_{L^\infty(\C_1)}+\|g\|_{C^{0,\alpha}(\Gamma)}+ C_{f^-}+C_{f^+} \big),
$$
where $C>0$ depends only on $n$, $\lambda$, $\Lambda$, and $\alpha$.

\end{Theorem}

Theorem~\ref{thm:globalest} will be a consequence of the following pointwise $\Coa$ estimates at the interface. 

\begin{Theorem}[Boundary pointwise $\Coa$ estimates] \label{Thm:main_boundary_Coa_regularity}
Let $0<\alpha<\bar\alpha$.
Assume that $0 \in \Gamma$, $\psi \in C^{1,\alpha}(0)$, $\psi\not\equiv 0$, $g \in C^{0,\alpha}(0)$, and $f^{\pm}$ satisfy that there are constants $C_{f^\pm}>0$ such that
\[ 
\left( \fint_{\Omega^{\pm}_r} |f^{\pm}|^{n+1} \, dxdt \right)^\frac{1}{n+1} \leq C_{f^\pm} r^{\alpha-1}, \quad \text{for all } r > 0 \text{ small}. 
\]
Suppose that $u$ is a bounded viscosity solution to \eqref{eq:main1}, with $\|u\|_{L^\infty(\C_1)} \leq 1$.
Then $u^{\pm} \in C^{1,\alpha}(0)$, i.e., there exist affine functions $l^{\pm}(x) = A^{\pm} \cdot x + b$ such that
\[ |u^{\pm}(x,t) - l^{\pm}(x)| \leq C r^{1+\alpha}, \quad \text{for all } (x,t) \in \Omega_{r}^{\pm}, \]
for all $r\leq r_0$, with $r_0 = C_0/\|\psi\|_{C^{1,\alpha}(0)}$, and $C_0 >0$, depending only on $n$, $\lambda$, $\Lambda$, and $\alpha$. Moreover,
\[ |A^-| + |A^+| + |b| + |C| \leq C_0 \|\psi\|_{C^{1,\alpha}(0)} \big(|g(0)| + [g]_{C^{0,\alpha}(0)} + C_{f^-} + C_{f^+}\big). \]
\end{Theorem}

The paper is organized as follows. In the remainder of Section~\ref{Sec:Intro}, we introduce some useful notation and definitions. Section~\ref{Sec:ABP} focuses on the ABP--Krylov--Tso estimate, which relies on a barrier that we construct using the Hopf lemma. In Section~\ref{Sec:Holder}, we establish the H\"older regularity for the $\mathcal{S}^*(f)$ class, which follows from a key Harnack inequality. Section~\ref{Sec:flat} is devoted to the flat interface problem (i.e., $\psi\equiv 0)$, for which we provide a comprehensive study. By considering Jensen's approximate solutions, we obtain a comparison principle, uniqueness, and $C^{1,\alpha}$ regularity. We also establish existence of viscosity solutions via Perron's method. In Section~\ref{Sec:Stab}, we present stability results connecting the curved interface problem to the flat one, showing that small perturbations in the interface yield small perturbations in the corresponding solutions. Finally, in Section~\ref{Sec:Coa}, we prove the piecewise $C^{1,\alpha}$ regularity up to the interface of the original solution, by importing the regularity of the solution to the flat problem through a geometric iteration scheme.

\subsection{Notation} \label{sec:notation}

We define the notation that we will use throughout the paper. For $(x,t)\in \R^{n+1}$, we call $x\in \mathbb{R}^n$ the space variable, and $t \in \R$ the time variable. 
 If $x=0$ and $t=0$, we simply write $0$ instead of $(0,0)$.
A point $x\in \mathbb{R}^n$ will sometimes be written as $x=(x',x_n)$, where $x'\in \R^{n-1}$ and $x_n \in \R$.
For a function $u(x,t)$, we denote by $\nabla u$ the gradient of $u$ with respect to $x$, $D^2 u$ the Hessian with respect to $x$, and $\nabla' u$ the gradient with respect to $x'$.

We call a \textit{parabolic cylinder} any set of the form 
 \[
 \mathcal{P}=U \times (t_1,t_2],
 \]
 where $U$ is a smooth bounded domain in $\R^n$ and $t_1 < t_2$. In particular, for $r>0$ and $(x_0,t_0)\in \R^{n+1}$, we define $\C_r(x_0,t_0)=B_r(x_0)\times (t_0-r^2,t_0]$ and $\C_r=B_r(0)\times (-r^2,0]$.  
 We denote by $\p_p \mathcal{P}$ the \textit{parabolic boundary} of $\mathcal{P}$ composed of the bottom and the lateral parts, respectively, i.e.,
 $$\p_p {\mathcal{P}}=\p_B\mathcal{P}\cup \p_L\mathcal{P}= \left( U\times \{t_1\}  \right)  \cup   \left( \p U \times [t_1,t_2] \right).$$
 For a general domain $\Omega\subset \R^{n+1}$, we define the parabolic boundary $\p_p\Omega$ as the set of points in the boundary $(x,t)\in \p\Omega$ such that, for any $\varepsilon>0$, it holds $\C_\varepsilon(x,t)\cap \Omega^c\neq \emptyset$.

 The interface $\Gamma:=\C_1\cap\, \{x_n=\psi(x',t)\}$ divides the unit cylinder $\C_1$ into two disjoint domains that we denote by $\Omega^\pm:=\C_1 \cap\{\pm(x_n-\psi(x',t))> 0\}$. When the interface is flat, i.e., $\psi\equiv 0$, we write $\C_1^\pm:=\C_1 \cap\{\pm \,x_n> 0\}$ and $\C_1'=\C_1\cap \{x_n=0\}$. For $r>0$, we define the rescaled sets $$\Omega_r^\pm=\Omega^\pm\cap \C_r \quad \text{ and } \quad  \Gamma_r=\Gamma\cap \C_r.$$ 
Given a function $v$, we denote the restriction of $v$ to $\Omega^\pm$ as $v^\pm=v\big|_{\Omega^\pm}$.

\subsection{Parabolic spaces}
 
Given $(x,t), (y,s)\in \R^{n+1}$, we define the \textit{parabolic distance} as
\[
d_p((x,t), (y,s)) :=\big(|x-y|^2 + |t-s|\big)^{1/2}.
\]
Let $\mathcal{P}$ be a parabolic cylinder as above. Let $C(\mathcal{P})$ be the set of continuous functions in ${\mathcal{P}}$. For $0<\alpha \leq 1$ and $u\in C(\mathcal{P})$, we define the parabolic H\"older semi-norm as
\begin{equation*}
    [u]_{C^{0,\alpha}(\mathcal{P})} := \sup_{\substack{(x,t)\neq (y,s)}} \frac{\nor{u(x,t)-u(y,s)}}{d_p((x,t),(y,s))^{\alpha}},
\end{equation*}
and the parabolic H\"older norm as
\begin{equation*}
    \norm{u}_{C^{0,\alpha}(\mathcal{P})} := \norm{u}_{L^{\infty}(\mathcal{P})} +[u]_{C^{\alpha}({\mathcal{P}})}.
\end{equation*}
Further, we define the H\"older semi-norm in $t$ as
\begin{equation*}
[u]_{C_t^{\alpha}({\mathcal{P}})} := \sup_{ \substack{(x,t)\neq (x,s)}} \frac{|u(x,t)-u(x,s)|}{|t-s|^\alpha}.
\end{equation*}
This is necessary to introduce the $\Coa$-norm in ${\mathcal{P}}$:
\begin{equation*}
    \norm{u}_{\Coa({\mathcal{P}})}:=\norm{u}_{L^{\infty}({\mathcal{P}})} + \norm{\nabla u}_{L^{\infty}({\mathcal{P}})} +[\nabla u]_{C^{0,\alpha}({\mathcal{P}})} + [u]_{C_t^{(1+\alpha)/2}({\mathcal{P}})}.
\end{equation*}

We denote by $C^{0,\alpha}({\mathcal{P}})$ (resp.~$\Coa({\mathcal{P}})$) the space of continuous (resp.~differentiable in space) functions in $\mathcal{P}$ which are bounded in the $\norm{\cdot}_{C^{0,\alpha}({\mathcal{P}})}$ norm (resp.~$\norm{\cdot}_{\Coa(\mathcal{P})}$).

\subsection{Viscosity solutions and the Pucci class} \label{sec:viscosity}
Next, we define the notion of viscosity solution to \eqref{eq:main1}. 
 \begin{Definition}
Given $u, v \in C(\mathcal{P})$, we say that $v$ touches $u$ from above at a point $(x_0,t_0)\in \mathcal{P}$ if $u(x_0,t_0)=v(x_0,t_0)$, and there is some $\delta>0$ such that $u \leq v$ in $\C_\delta(x_0,t_0) \subseteq \mathcal{P}.$
 \end{Definition}

 We denote by $\usc(\mathcal{P})$ (resp.~$\lsc(\mathcal{P})$) the set of upper (resp.~lower) semi-continuous functions.

\begin{Definition} 
We say that $u \in \usc(\C_1)$ is a viscosity subsolution to \eqref{eq:main1} if whenever a continuous function $\varphi$ touches $u$ from above at $(x_0,t_0)$ the following holds, for some small $\delta>0$:
\begin{itemize}
    \item[$(i)$] If $(x_0,t_0) \in \Omega^{\pm}$ and $\varphi \in C^{2}(\C_\delta(x_0,t_0) )$, then 
    \[ 
    \p_t \varphi(x_0,t_0) - F^{\pm}(D^2 \varphi(x_0,t_0)) \leq f^{\pm}(x_0,t_0). 
    \]
    \item[$(ii)$] If $(x_0,t_0) \in \Gamma$ and  $\varphi \in C_x^1 \big( \C_\delta(x_0,t_0) \cap \overline{\C_1^+} \big) \cap  C_x^1 \big(\C_\delta(x_0,t_0) \cap \overline{\C_1^-}\big)$, then 
    \[
\varphi_\nu^+(x_0,t_0)-\varphi_\nu^-(x_0,t_0)\geq g(x_0,t_0).
    \]
\end{itemize}
 Similarly, a function $u\in \lsc(\C_1)$ is a viscosity supersolution to \eqref{eq:main1} if whenever $\varphi$ touches $u$ from below at $(x_0,t_0)\in \C_1$, then $\varphi$ satisfies $(i)$ and $(ii)$, where all inequalities are reversed. We say that $u\in C(\C_1)$ is a viscosity solution to \eqref{eq:main1} if it is both a viscosity subsolution and supersolution.
\end{Definition}

We also recall the definition of the extremal Pucci operators with ellipticity constants $\lambda$ and $\Lambda$:  
\begin{equation*}
\mathcal M_{\lambda,\Lambda}^+ (M) = \sup_{\lambda I \le A \le \Lambda I}  \tr (A M), \quad \quad  \quad \mathcal M_{\lambda,\Lambda}^- (M) = \inf_{\lambda I \le A \le \Lambda I}  \tr (A M).
\end{equation*} 
For simplicity, we will just call $\Mp$ and $\Mm$ when the ellipticity constants are understood. The class $\underline{\mathcal{S}}(f^\pm)$ denotes the functions $u \in \usc(\C_1)$ such that $\p_t u - \Mp (D^2 u) \leq f^\pm $ in $\Omega^{\pm}$. Furthermore, we say $u\in \overline{\mathcal{S}}(f^\pm)$ if $-u\in \underline{\mathcal{S}}(-f^\pm)$ and $S (f^\pm)=\underline{\mathcal{S}}(f^\pm)\cap \overline{\mathcal{S}}(f^\pm)$. Finally, $\mathcal{S}^{*}(f^\pm)=\underline{\mathcal{S}}(-|f^\pm|)\cap \overline{\mathcal{S}}(|f^\pm|)$.

\section{ABP-Krylov-Tso estimate}\label{Sec:ABP}

A key tool to study the regularity of viscosity solutions to \eqref{eq:main1} is the ABP-Krylov-Tso estimate.
Given a continuous function $u$ in $\C_1$, we define its the convex envelope as
\begin{equation}\label{Def:convex_env}
    \begin{aligned}
        C_u(x,t):=&\,\sup\{ v(x,t): v \leq u \text{ 
 in }\C_1,\ v(y,s) \text{ convex in } y \text{ and decreasing in }s\} \\
=&\, \sup \{  A \cdot x + b : A\cdot y + b \leq u(y,s) \ \text{ for all }  s \leq t \text{ and } y \in B_1\},
    \end{aligned}
\end{equation}
(e.g., see \cite[Definition~3.13]{Wang1992} and \cite[Lemma~4.2]{imbert2013introduction}). We denote the positive and negative parts of a function $v$ as  $v_{+}= \max \{0,v \}$ and $v_- = -\min\{0,v\}$, respectively. 
In the following, $C_u$ denotes the convex envelope of $-u_-$ in $\C_2$, where we have extended $u$ by zero outside $\C_1$. 

\begin{Theorem}[ABP-Krylov-Tso estimate]\label{Thm:ABP}
Let $f \in C(\C_1\setminus \Gamma) \cap L^{\infty}(\C_1)$, $g\in L^\infty(\Gamma)$, and $\psi \in C^{1,\alpha}(\overline{\C_1'})$, for some $0<\alpha<1$. Let $u$ satisfy 
    \begin{align*}
        \begin{cases}
            u\in \overline{\mathcal{S}}(f^\pm) & \text{ in } \Omega^\pm,\\
        u_\nu^+-u_\nu^- \leq g &\text{ on } \Gamma,
        \end{cases}
    \end{align*}
    in the viscosity sense.
 Then 
 \[ 
 \sup_{\C_1} u_- \leq \sup_{\p_p \C_1} u_- + C \Big(\max_\Gamma g_+ + \| f_+\|_{L^{n+1}(\C_1)}\Big),
  \]
where $C>0$ depends only on $n$, $\lambda$, $\Lambda$, $\alpha$, and $\|\psi\|_{C^{1,\alpha}(\overline{\C_1'})}$.
\end{Theorem}

 As an immediate consequence, we get the maximum principle.

\begin{Corollary}[Maximum principle] \label{cor:MP}
    Let $u$ satisfy 
    \begin{align*}
        \begin{cases}
            u\in \overline{\mathcal{S}}(0) & \text{ in } \Omega^\pm,\\
        u_\nu^+-u_\nu^- \leq 0 &\text{ on } \Gamma,
        \end{cases}
    \end{align*}
    in the viscosity sense. If $u \geq 0$ on $\partial_p \C_1$, then $u\geq 0$ in $\C_1$.
\end{Corollary}

To prove Theorem~\ref{Thm:ABP}, we introduce a useful barrier.

\begin{Lemma}[Barrier]\label{Lm:Barrier11}
    Let $\Omega\subset \R^{n+1}$ be a $\Coa$ domain, for some $0<\alpha<1$, such that $0\in \p_p\Omega$. Call $\Gamma=\p_p \Omega\cap \C_2$. Let $w$ be the viscosity solution to
    \begin{align}\label{eq:barrier11}
        \begin{cases}
            \p_t w-\Mm(D^2w)=0 & \text{ in } \Omega\cap \C_2,\\
            w=\phi & \text{ on } \Gamma,\\
            w=1 & \text{ on } \p_p(\Omega\cap \C_2)\setminus \Gamma,
        \end{cases}
    \end{align}
    where $\phi\in C^\infty(\Gamma)$ satisfies $0\leq \phi\leq 1$, $\phi\equiv 0$ on $\Gamma\cap \C_1$, and $\phi\equiv 1$ on $\Gamma\setminus\C_{3/2}$.

    Then $w$ is a classical solution to \eqref{eq:barrier11} in $\Omega\cap \C_1$, with $0\leq w\leq 1$, and
    \begin{align}\label{eq:barrier_hopf}
        w_\nu\geq c_0>0 \quad \text{ on } \Gamma\cap \C_1,
    \end{align}
    where $\nu$ is the (spatial) interior unit vector of $\Gamma$ and $c_0>0$ depends on $n$, $\lambda$, $\Lambda$, $\alpha$ and $[\Gamma]_{\Coa}$.
    
\end{Lemma}
\begin{proof}
    The existence, uniqueness, and interior $\Cta$ regularity of solutions to \eqref{eq:barrier11} follow from the results in \cite{Wang1992,Wang1992II}. Furthermore, since $\phi$ is smooth on $\Gamma$, we also have $\Coa$ regularity up to $\Gamma\cap \C_1$. By the classical maximum principle, $0\leq w\leq 1$. Hence, we only need to prove \eqref{eq:barrier_hopf}.

    Let $(x_0,t_0)\in \Gamma\cap \C_1$. After a translation, rotation, and rescaling, we can assume without loss of generality that $x_0=0$, $t_0=0$, and
    \[
    \Omega\cap \C_1= \{(x,t)\in \C_1 \,:\,x_n>\psi(x',t)\},
    \]
    for some $\psi\in \Coa(\overline{\C'_1})$, with $\nabla'\psi(0)=0$, and $[\psi]_{C^{1,\alpha}(0)}\leq 1/4$. Then $\Omega$ satisfies the interior $C^{1,\dini}$ condition with $\omega_\Omega(r)=\tfrac{1}{4} r^{\alpha}$ and $\bar r=1$ (see Definition~\ref{def:intdini}). 
    Moreover, $w$ satisfies the assumptions of the Hopf Lemma (Theorem~\ref{thm:hopf}) in $\Omega\cap\C_1$. Hence, by setting $l=\nu(0)=e_n$, we get 
    \[
    w(re_n,0)\geq c w(e_n/2,-3/4)\,r\quad \text{ for all } 0<r<r_0.
    \]
    Since $w$ is differentiable at the origin, letting $r\to 0$, we get $w_\nu(0,0)\geq c w(e_n/2,-3/4)$. To bound $w(e_n/2,-3/4)$ universally from below, we define 
    \[
    \widetilde w(x,t)=w(x,t)-\tfrac12x_n+\tfrac18 \quad \text{ in } D:=\{(x,t)\in \Omega\cap \C_2\, :\, x_n> 1/4\}.
    \]
    Since $|\psi|\leq 1/4$, then $\Gamma\cap D=\emptyset$. Hence, $\widetilde w\in \overline {\mathcal{S}}(0)$ in $D$. We claim that $\widetilde w\geq 0$ on $\p_pD$. Indeed, note that $\p_pD= (\C_2\cap \{x_n=1/4\}) \cup (\partial_p(\Omega\cap \C_2)\cap \{x_n>1/4\}).$ On $\C_2\cap \{x_n=1/4\}$, we have $w\geq 0$, thus $\widetilde w \geq - 1/8+1/8=0$. On $\partial_p(\Omega\cap \C_2)\cap \{x_n>1/4\}$, we have $w=1$, thus $\widetilde w \geq 1 -1+1/8=1/8$.
    By the maximum principle, it follows that $\widetilde w\geq 0$ in $D$, which implies $w(e_n/2,-3/4)\geq 1/8$. 
    
    Therefore $w_\nu(0,0)\geq c_0$, with $c_0=c/8>0$.
\end{proof}

We are ready to give the proof of the ABP-Krylov-Tso estimate. 
\begin{customproof}{Theorem~\ref{Thm:ABP}}
Let $\Gamma_2 = \C_2\cap \{x_n=\psi(x',t)\}$ and $\Omega_2^\pm = \C_2\cap \{\pm(x_n-\psi(x',t))>0\}$. After a translation, we may assume that $\psi(0)=0$, so that $0\in \partial_p \Omega_2^\pm$. Let $w^\pm$ be the  barriers constructed in Lemma~\ref{Lm:Barrier11} in the domains $\Omega_2^\pm$.
Note that $w^+=w^-$ on $\Gamma_2$, thus the function $w=w^+\chi_{\overline\Omega_2^+}+w^-\chi_{\Omega_2^-}$, defined in $\C_2$, is continuous across $\Gamma_2$.
We have $0\leq w\leq 1$ in $\Omega^\pm$, and 
\[
w_\nu^+\geq c^+>0,\quad w_\nu^-\leq -c^-<0 \quad \text{ on } \Gamma,
\]
where $c^\pm>0$ depend only on $n$, $\lambda$, $\Lambda$, $\alpha$, and $[\psi]_\Coa$.

Fix $\varepsilon>0$ small, and consider
\[
v=u-\frac{1}{c_0}\left(\max_\Gamma g_+ + \varepsilon\right)w \quad \text{ in } \C_1,´
\]
with $c_0=c^++c^-$. Then $v\in \overline{\mathcal{S}}(f^\pm)$ in $\Omega^\pm$.
Moreover, $v$ satisfies in the viscosity sense, 
\begin{align}\label{eq:tcneg}
    v_\nu^+-v_\nu^-\leq&\, g-\frac{1}{c_0}\left(\max_\Gamma g_++\varepsilon\right)(w_\nu^+-w_\nu^-)
    \leq  -\varepsilon \quad \text{ on } \Gamma.
\end{align}

Up to considering $v-\inf_{\p_p \C_1}v$, we can suppose that $v\geq 0$ on $\p_p \C_1$. Assume also that $v_-\not\equiv0$, otherwise the result is trivial. Let $C_v$ be the convex envelope of $-v_-$ in $\C_2$, where we have extended $v$ by zero outside $\C_1$.
Let $t_0$ be the last time where $v_-$ vanishes for all times up to $t_0$, i.e.,
\[
t_0 :=\sup\Big\{t\in[-2,0] : \sup_{B_1\times[-2,t)}v_-=0\Big\}.
\]
Since $v_-\not\equiv 0$, then $t_0\in(-1,0)$. Hence, $C_v(x,t)<0$ for every $(x,t)\in B_2\times(t_0,0]$, and we get
\[
(\p B_1\times(t_0,0] )\cap \{v=C_v\}=\emptyset.
\]
We claim that $\{(x,s)\in \Gamma\,:\, s\geq t\}\cap\{v=C_v\}=\emptyset$ for all $t>t_0$. Otherwise, if $(\bar x,\bar t)\in \Gamma\cap\{v=C_v\}$ with $\bar t >t_0$, then there exists a supporting plane $A\cdot x+b$ which touches the graph of $v$ from below at $(\bar x, \bar t)$, for some $A\in \R^n$ and $b\in \R$. Hence, by the viscosity condition on $\Gamma$ in \eqref{eq:tcneg}, we get
\[
-\varepsilon\geq A\cdot \nu(\bar x, \bar t)-A\cdot \nu(\bar x, \bar t)=0,
\]
which is a contradiction.
Moreover, since $\{v=C_v\}$ and $\{(x,s)\in \Gamma\,:\, s\geq t\}$ are 
disjoint closed sets for all $t>t_0$, the parabolic distance between them is strictly positive.

Next, we show that $C_v\in C^{1,1}(\overline B_1\times[t_1,0])$, for all $t_1>t_0$, by showing that there exist $K>0$ and $0<r\leq 1$ such that for all $(\bar x, \bar t)\in \overline B_1\times [t_1,0]\cap \{v=C_v\}$, there exists a convex parabolic paraboloid of opening $K$ that touches $C_v$ by above at $(\bar x,\bar t)$ in $\C_r(\bar x,\bar t)$. Indeed, let $t_1>t_0$ and take $(\bar x, \bar t)\in ( \overline B_1\times [t_1,0])\cap \{v=C_v\}$. Since $(\bar x, \bar t)\not\in \p_p \C_1\cup \Gamma$, we may assume that $\C_\delta(\bar x,\bar t)\subset \Omega^+$, for $\delta$ sufficiently small.
Let $\ell(x)$ be a supporting plane of $C_v$ at $(\bar x, \bar t)$. Then 
\[
0\leq C_v-\ell\leq -v_--\ell\quad \text{ in } \C_\delta(\bar x, \bar t),
\]
with equality at $(\bar x,\bar t)$. Furthermore, $-v_--\ell\in \overline {\mathcal{S}}(f^+)$ in $\C_\delta(\bar x, \bar t)$. Hence, by \cite[Lemma 3.16]{Wang1992}, 
\begin{align}\label{eq:convex_env}
    C_v(x,t)\leq \ell(x)+C^+\Big(\sup_{\C_\delta(\bar x, \bar t)}f^+_+\Big)\left(|x-\bar x|^2-(t-\bar t)\right)\text{ in } \C_{\delta\gamma^+}(\bar x, \bar t),
\end{align}
where $\gamma^+<1$ and $C^+$ are universal constants. 
If $(\bar x, \bar t)\in \Omega^-\cap \{v=C_v\}$, the proof is identical. Hence, we take $K=2\max\{C^+,C^-\}\|f_+\|_\infty$ and $r=\delta \min\{\gamma^+,\gamma^-\}$.

Therefore, there exists a set $E\subseteq \overline B_1\times [t_1, 0]$ such that $|(\overline B_1\times [t_1, 0])\setminus E|=0$, and $C_v$ is twice differentiable w.r.t $x$ and once w.r.t. $t$, for all $(x,t)\in E$. Moreover, we have
\begin{align*}
    \sup_{B_1\times [t_1, 0]} v_-\leq C\Big( \int_{E\cap \{v=C_v\}} (-\p_t C_v) \det D^2 C_v\, dxdt   \Big)^{\frac{1}{n+1}}.
\end{align*}
Letting $\delta \to 0$ in \eqref{eq:convex_env}, and exploiting the continuity of $f_+$, we obtain
\[
(-\p_t C_v) \det D^2 C_v\leq Cf_+^{n+1} \quad \text{ in } E\cap \{v=C_v\}.
\]
Combining both estimates, for any $t_1>t_0$, it follows that
$$
    \sup_{B_1\times [t_1, 0]} v_-\leq C\|f_+\|_{L^{n+1}(\C_1)}.
$$
Since $v_-\equiv 0$ on $B_1\times [-1, t_0]$, and $C$ is independent of $t_1$, letting $t_1 \to t_0$, the estimate holds on $\C_1$. 

Finally, letting $\vep\to0$, we obtain the desired result for $u$.
\end{customproof}

\begin{Remark}\label{Rmk:ABP_flat}
    As in the elliptic case (see \cite[Remark 2.4]{soria2023regularity}), when the interface is flat, i.e., $\Gamma=\C_1'$, we obtain the improved estimate
    \[
    \sup_{\C_1}u_-\leq \sup_{\p_p\C_1}u_-+C\left(\max_{\Gamma}g_++\|f_+\|_{L^{n+1}(\C_1\cap \{u=C_u\})}\right).
    \]
\end{Remark}

\section{H\"older regularity}\label{Sec:Holder}

In this section, we prove interior and global H\"older estimates for functions in the Pucci class satisfying the transmission condition in the viscosity sense.
The interior H\"older estimate (Theorem~\ref{thm:intholder}) follows from a standard iteration of the next oscillation decay.

\begin{Theorem}[Oscillation decay] \label{lem:oscdecay}
Let $u$ be as in Theorem~\ref{thm:intholder}. There is a dimensional constant $r_0<1/4$ such that
\[
\osc_{\C_{r_0}}u\leq \mu \displaystyle\osc_{\C_1}u+C\left(\|g\|_{L^\infty(\Gamma)} + \|f\|_{L^{n+1}(\C_1)}\right),
\]
where $0<\mu<1$ and $C>0$ depend only on $n$, $\lambda$, $\Lambda$, $\alpha$, and $\|\psi\|_{C^{1,\alpha}(\overline{\C_1'})}$.
\end{Theorem}

First, we need two preliminary lemmas. The following barrier was constructed in \cite[Lemma~3.22]{Wang1992} (see also \cite[Lemma~4.16]{imbert2013introduction}).

\begin{Lemma} \label{Lm:Barrier_Wang}
    Let $\mathcal{P}=B_1\times(0,1]$, $K^1=(-r,r)^n\times(0,r^2]$, and $K^2=(-3r,3r)^n\times(r^2,10r^2]$, for $r<\tfrac{1}{3\sqrt{n}}$. Then there is a function $\phi\in C^{1,1}(\P)$ such that
    \begin{align*} 
        \begin{cases}
            \phi(x,t)\geq 1 &\text{ in } K^2,\\
            \phi\leq 0 &\text{ on } \p_p \P,\\
            \p_t \phi-\Mm(D^2\phi)\leq 0 &\text{ in } \P\setminus K^1.
        \end{cases}
    \end{align*}
    Moreover, there is $C_0>0$, depending only on $n$, $\lambda$, $\Lambda$, and $r$, such that
    \[
    \|\phi\|_{C^{1,1}(\P)}\leq C_0.
    \]
\end{Lemma}

The main lemma is the following Harnack inequality.

\begin{Lemma}[Harnack inequality]\label{Lm:Harnack_new}
Let $u$ be as in Theorem~\ref{thm:intholder}. Take $r= \tfrac{1}{4\sqrt n}$ and $\sigma = \tfrac{1}{2(1+2r)}$. Assume that $\|\psi\|_{L^\infty(\C_1')}\leq \tfrac{\sigma r}{2}$, $\|u\|_{L^\infty(\C_1)}\leq 1$, and 
$$
 u(\bar x , \bar t) \geq 0 \quad \text{ with } \quad (\bar x, \bar t)=(2\sigma r e_n, -12\sigma^2 r^2).
$$
There exist $0<\vep_0,c <1$, depending only on $n$, $\lambda$, $\Lambda$, $\alpha$, and $[\psi]_{C^{1,\alpha}(\overline{\C_1'})}$, such that if
$$\|g\|_{L^\infty(\Gamma)} + \|f\|_{L^{n+1}(\C_1)}\leq \vep_0,$$
then there is a dimensional constant  $r_0<\tfrac{1}{2(1+2\sqrt{n})}$ such that
\begin{equation*} 
\inf_{\C_{r_0}} u \geq -1 + c.
\end{equation*}
\end{Lemma}

\begin{proof}
Fix $0<\vep_0<1$ to be chosen small.
Set $r= \tfrac{1}{4\sqrt n}$ and $\sigma = \tfrac{1}{2(1+2r)}$.
 Let $\bar x=2\sigma r e_n$, $\bar t=-12\sigma^2 r^2$, and $\tilde t= \bar t+2\sigma^2r^2=-10\sigma^2r^2$.
Define the cylinders (see Figure~\ref{fig:harnack}):
\begin{align*}
\widetilde \P & := B_\sigma\times (0,\sigma^2]+(\bar x, \tilde t),\\
\widetilde K^1& :=(-\sigma r, \sigma r)^n\times (0,\sigma^2r^2]+ (\bar x, \tilde t),\\ \widetilde K^2 & :=(-3\sigma r, 3\sigma r)^n\times (\sigma^2r^2,10\sigma^2r^2]+ (\bar x, \tilde t),\\
\widetilde K^3 &:= B_{\tfrac{\sigma r}{4}}\times(-\tfrac{\sigma^2 r^2}{4},0] + (\bar x, \bar t). 
\end{align*}
Note that since $\|\psi\|_{L^\infty(\C_1')}\leq \tfrac{\sigma r}{2}$, then ${\widetilde K^1}\subset \subset\Omega^+$. Moreover, $\widetilde K^2\cap \Omega^-\neq \emptyset$.

\begin{figure}[htbp]
  \centering

\begin{tikzpicture}[scale=4]

% Box Q tilde
\draw[thick] (-0.3,-0.45) rectangle (0.9,0.15);
\node at (0.8,0.24) {$\widetilde{\P}$};

% Box C1
\draw[thick] (-1,-1.1) rectangle (1,0);
\node at (-0.5,0.08) {$\C_1$};

% Coordinate labels
\node at (-1.05,-1.1) {$-1~~~$};
\node at (-1.05,-0.441) {$\tilde t~-$};
\node at (-1.05, -0.73) {$\bar t~-$};
\node at (0,-1.21) {$\sigma r$};
\node at (0.3,-1.21) {$\bar{x}$};
\node at (0,0.06) {$0$};

%\fill (0.3,-0.45) circle[radius=0.42pt];
\fill (0.3,-0.75) circle[radius=0.42pt];

\draw[thin] (0.3,-1.14) -- (0.3,-1.06);
\draw [decorate,decoration={brace,mirror,amplitude=5pt}, thick]
  (-0.1,-1.1) -- (0.1,-1.1) ;

% Vertical curly line (axis)
\draw[thick, decorate, decoration={snake, segment length=15pt, amplitude=3pt}] (0,-1.1) -- (0,0);

% Dashed lines
\draw[dashed] (-0.1,-1.1) -- (-0.1,0);
\draw[dashed] (0.1,-1.1) -- (0.1,0);

% Green shaded region and label C_r0
\draw[green, thick] (-0.15,0) rectangle (0.15,-0.15);
\fill[green!20, opacity=0.8] (-0.15,0) rectangle (0.15,-0.15);
\node[green!50!black] at (0.05,-0.075) {$\C_{r_0}$};

%Blue box K^1
\draw[blue, thick] (0.15,-0.45) rectangle (0.45,-0.3);
\node[blue] at (0.3,-0.38) {$\widetilde{K}^1$};

% Blue box K^2
\draw[blue, thick] (-0.25,-0.3) rectangle (0.8,0);
\node[blue] at (0.3,-0.15) {$\widetilde{K}^2$};

% Small box K^3
\draw[thick] (0.225,-0.9) rectangle (0.375,-0.75);
\node at (0.3,-0.825) {$\widetilde{K}^3$};

%Curved domains \Omega^\pm and \widetilde P in \C_1
\node at (0.7,-0.9) {$\Omega^+$};
\node at (-0.7,-0.9) {$\Omega^-$};
\node at (0.7,-0.38) {$\widetilde\P_1$};
\end{tikzpicture}

  \caption{The cylinders defined in the proof.}
  \label{fig:harnack}
\end{figure}

By the interior parabolic Harnack inequality applied to $u+1\geq0$, 
$$
\sup_{\widetilde K^3} (u+1) \leq C \Big( \inf_{\widetilde K^1} (u+1) + \|f^+\|_{L^{n+1}(\Omega^+)}\Big), 
$$
where $C>0$ depends only on $n$, $\lambda$, and $\Lambda$. Since  $u(\bar x, \bar t)\geq 0$ and $\|f\|_{L^{n+1}(\C_1)}\leq \vep_0$, we get
\begin{equation} \label{eq:intharnack}
    \inf_{\widetilde K^1} u\geq -1+ \widetilde c,
\end{equation}
where $\widetilde c= \frac{1}{C}-\vep_0>0$, taking $\vep_0<1/C$.  

Let $\phi$ be as in Lemma~\ref{Lm:Barrier_Wang}, and consider the function
$$
\widetilde \phi (x,t) = \phi \Big(\frac{x-\bar x}{\sigma}, \frac{t- \tilde t}{\sigma^2} \Big) \quad \text{ for } (x,t) \in \widetilde \P.
$$ 
Then $\widetilde \phi$ satisfies $\|\widetilde\phi\|_{C^{1,1}(\widetilde \P)}\leq C_0$, and
    \begin{align*}
        \begin{cases}
            \widetilde\phi(x,t)\geq 1 &\text{ in } \widetilde K^2,\\
            \widetilde\phi\leq 0 &\text{ on } \p_p \widetilde \P,\\
            \p_t \widetilde\phi-\Mm(D^2\widetilde\phi)\leq 0 &\text{ in } \widetilde \P\setminus \widetilde K^1.
        \end{cases}
    \end{align*}
Let $\widetilde \P_1= \widetilde \P \cap \C_1$. For $\eta>0$ small to be determined, let
$$
    v=\eta~\Big( \widetilde\phi- \frac{1}{2}\Big)+\frac{\varepsilon_0}{c_0}w \quad \text{ in } \widetilde \P_1,
$$
where $w$ and $c_0$ are given in the proof of Theorem \ref{Thm:ABP}. By construction,
$$
    \p_t v-\Mm(D^2v)\leq \eta \big(\p_t \widetilde\phi-\Mm(D^2 \widetilde\phi) \big)\leq 0 \quad \text{ in }  (\widetilde \P_1\setminus \widetilde K_1)\cap \Omega^\pm.
$$
Since $\widetilde\phi$ is differentiable across $\Gamma$,
    \begin{align*}
        v_\nu^+-v_\nu^-=\frac{\varepsilon_0}{c_0}(w_\nu^+-w_\nu^-)\geq 2\varepsilon_0>\|g\|_\infty\geq g \quad \text{ on } \Gamma\cap \widetilde \P_1.
    \end{align*}
    We will choose $\vep_0$ and $\eta$ such that $v \leq u+1$ on $\partial_p (\widetilde \P_1 \setminus \widetilde K^1)=(\partial_p \widetilde \P_1 \cup \partial \widetilde K^1) \setminus \partial_B \widetilde K^1$, where  $\partial \widetilde K^1$ denotes the topological boundary of $\widetilde K^1$. Indeed, since $\widetilde\phi\leq 0$ and $w\leq 1$ on $\p_p \widetilde \P_1$, then
        $$
    v\leq - \frac{\eta}{2} + \frac{\varepsilon_0}{c_0} \leq 0\quad  \text{ on } \p_p \widetilde \P_1,
    $$
    provided $\vep_0 \leq \frac{\eta c_0}{2}$.  In this case, we also have
    $$
    v \leq \eta \|\widetilde \phi\|_{L^\infty(\partial \widetilde K^1)} + \frac{\varepsilon_0}{c_0} \leq  \eta \Big(C_0+\frac{1}{2}\Big) \quad \text{ on } \partial \widetilde K^1.
    $$
    First, choose $\eta\leq \frac{\widetilde c }{C_0 +1/2}$, and then, choose $\vep_0 \leq \frac{\eta c_0}{2}$. It follows that $v\leq \widetilde c$ on $\partial \widetilde K^1$. Since $u+1\geq 0$ in $\widetilde \P_1$ and $ u+1\geq \widetilde c$ on $\widetilde K_1$ by \eqref{eq:intharnack}, we have
    \begin{equation}\label{eq:bdrycond}
    v\leq u+1\quad \text{ on } \partial_p (\widetilde \P_1 \setminus \widetilde K^1).
    \end{equation}
    
Since  $u  \in \mathcal{S}^*(f^{\pm})$ in $\Omega^\pm$ and $v \in \underline{\mathcal{S}}(0)$ in $(\widetilde \P_1 \setminus \widetilde K^1)\cap \Omega^\pm$,  noting that $v\in C^{1,1}(\widetilde\P_1\cap \Omega^\pm)$ and using \cite[Lemma~3.12]{Wang1992}, we have $u+1-v\in \overline{\mathcal{S}}(|f^\pm|)$ in $(\widetilde \P_1 \setminus \widetilde K^1)\cap \Omega^\pm$. Also, the following transmission condition holds in the viscosity sense:
\[
\p_\nu(u+1-v)^+-\p_\nu(u+1-v)^-\leq 0\quad  \text{ on } \Gamma \cap \widetilde \P_1.
\]
Hence, applying Theorem \ref{Thm:ABP} to $u+1-v$ in $\widetilde \P_1\setminus \widetilde K^1$, and using \eqref{eq:bdrycond}, we get
\[
\sup_{\widetilde \P_1\setminus \widetilde K^1} (u+1-v)_-\leq \sup_{\p_p(\widetilde \P_1\setminus \widetilde K^1)}(u+1-v)_-+C\|f\|_{L^{n+1}(\C_1)}\leq C\varepsilon_0.
\]
Therefore $u\geq -1+v-C\varepsilon_0$ in $\widetilde \P_1\setminus \widetilde K^1$. Moreover, since $\widetilde \phi\geq 1$ in $\widetilde K^2$, then $v\geq \eta$ in $\widetilde K^2$,
and thus,
\begin{align*} 
    \inf_{\C_{r_0}} u\geq -1+c,
\end{align*}
where $c=\eta-C\varepsilon_0>0$, taking $\vep_0$ sufficiently small, and $r_0=\sigma r< \frac{1}{2(1+2\sqrt{n})}$.
\end{proof}

\begin{customproof}{Theorem~\ref{lem:oscdecay}}
Let $\sigma$, $r$, and $\vep_0$ be as in Lemma~\ref{Lm:Harnack_new}.
After a translation, rotation and rescaling, we may assume that $\psi(0)=0$, $\nabla' \psi(0)=0$, and
\[
[\psi]_{\Coa(0)}:= \sup_{(x',t)\in \C_1'\setminus\{0\}} \frac{|\psi(x',t)-\psi(0)-\nabla'\psi(0)\cdot x'|}{(|x'|^2+|t|)^{(1+\alpha)/2}}\leq \frac{\sigma r}{2}.
\]
Then $|\psi|\leq  \frac{\sigma r}{2}$ in $\C_1'$.
 Let $M = \|g\|_{L^\infty(\Gamma)} + \|f\|_{L^{n+1}(\C_1)}$, and consider
\[
\widetilde{u} = \frac{2u - (\inf_{\C_1} u + \sup_{\C_1} u)}{\mathrm{osc}_{\C_1} u + 2M/\varepsilon_0} \quad \text{ in } \C_1.
\]
We have $\widetilde u\in \mathcal{S}^*(\widetilde{f}^\pm)$ in $\Omega^\pm$, with
$
\widetilde{f}^\pm= 2f^\pm(\mathrm{osc}_{\C_1} u + 2M/\varepsilon_0)^{-1}.
$
Also, $\widetilde{u}_\nu^+ -  \widetilde{u}_\nu^- \le \widetilde{g}$ on $\Gamma$, in the viscosity sense, with
$
\widetilde{g} = 2 g(\mathrm{osc}_{\C_1} u + 2M/\varepsilon_0)^{-1}.
$
Note that 
$$\|\widetilde{u}\|_{L^\infty(\C_1)} \le 1 \quad  \text{and} \quad
\| \widetilde{g}\|_{L^\infty(\Gamma)} + \|\widetilde{f}\|_{L^{n+1}(\C_1)}  \le \varepsilon_0.$$

Let $(\bar x, \bar t)=(2\sigma r e_n, -12\sigma^2 r^2)$. If $\widetilde{u}(\bar{x},\bar t) \ge 0$ then, by Lemma \ref{Lm:Harnack_new}, there is a dimensional constant $r_0>0$ and a constant $c>0$, depending only on $n$, $\lambda$, $\Lambda$, $\alpha$, and $[\psi]_{C^{1,\alpha}(\overline{\C_1'})}$, such that
\[
\inf_{\C_{r_0}} \widetilde{u} \ge -1 + c.
\]
Otherwise, $\widetilde{u}(\bar{x},\bar t) < 0$, and applying the lemma to $-\widetilde{u}$, we get 
\[
\sup_{\C_{r_0}} \widetilde{u} \le 1-c.
\]
Either way, we have $\osc_{\C_{r_0}}\widetilde u\leq 2-c,$
which in terms of $u$ reads
\[
\osc_{\C_{r_0}}u\leq \mu \osc_{\C_1}u+C\left(\|g\|_{L^\infty(\Gamma)} + \|f\|_{L^{n+1}(\C_1)}\right),
\]
with $\mu = \frac{2-c}{2} < 1$ and $C = \frac{2}{\varepsilon_0}$.
\end{customproof}

If the support of $g$ is compactly contained in $\Gamma$, we obtain the following global H\"older regularity result. 
The proof is standard and we omit it here (e.g., see \cite[Proposition~4.12]{Caffarelli-Cabre}).

\begin{Proposition}[Global H\"older regularity]
\label{Pro:Holder_global} 
Let $\alpha_1$ be as in Theorem \ref{thm:intholder}. Let $u \in C(\overline\C_1)$ satisfy
\[
\begin{cases}
u \in \mathcal{S}(0) & \text{in } \Omega^\pm, \\
u_\nu^+ - u_\nu^- = g & \text{on } \Gamma, \\
u = \varphi & \text{on } \p_p \C_1,
\end{cases}
\]
where $g \in L^\infty(\Gamma)$, with $\mathrm{supp}(g) \subset \Gamma \cap \C_{1-2\rho}$, for some $0 < \rho < 1/4$, $\varphi \in C^{0,\alpha}(\p_p \C_1)$, and $\psi \in C^{1,\alpha}(\overline\C_1)$, for some $0 < \alpha < 1$. Then $u \in C^{0,\beta}(\overline\C_1)$, with $0 < \beta \le \min\{\alpha_1, \alpha/2\}$, and
\[
\|u\|_{C^{0,\beta}(\overline\C_1)} \le \frac{C}{\rho^\gamma} \left( \|\varphi\|_{C^{0,\alpha}(\p_p \C_1)} + \|g\|_{L^\infty(\Gamma)} \right),
\]
 where $\gamma = \max\{\alpha, \alpha_1\}$, and $C>0$ depends only on $n$, $\lambda$, $\Lambda$, $\alpha$, and $\|\psi\|_{C^{1,\alpha}(\overline{\C_1'})}$.
\end{Proposition}

\section{The flat interface problem}\label{Sec:flat}

In this section, we study transmission problems with flat interfaces. Namely,
\begin{equation} \label{eq:flatpb}
\begin{cases}
    \partial_t u - F^\pm(D^2u)=f^\pm & \text{ in } \C_1^\pm,\\
    u_{x_n}^+-u_{x_n}^- = g & \text{ on } \C_1',
    \end{cases}
\end{equation}
where $\C_1^\pm = \C_1 \cap \{\pm\, x_n >0\}$ and $\C_1'=\C_1\cap \{x_n=0\}$.

\begin{Definition}
    We say that a function $u\in \usc(\C_1)$ is a viscosity subsolution to \eqref{eq:flatpb} if for any $\varphi$ touching $u$ from above at $(x_0,t_0)\in \C_1$, the following holds, for some small $\delta>0$:
    \begin{enumerate}
        \item[$(i)$] If $(x_0,t_0)\in \C_1^\pm$ and $\varphi \in C^2(\C_\delta(x_0,t_0))$, then
        $$
        \partial_t \varphi (x_0,t_0) - F^\pm (D^2\varphi(x_0,t_0)) \leq f^\pm (x_0,t_0).
        $$
        \item[$(ii)$] If $(x_0,t_0)\in \C_1'$ and $\varphi \in C^2 ( \overline{\C_\delta^-}(x_0,t_0))\cap C^2(\overline{\C_\delta^+}(x_0,t_0))$, then
        $$
        \varphi_{x_n}^+(x_0,t_0) - \varphi_{x_n}^-(x_0,t_0) \geq g(x_0,t_0),
        $$
        where $\varphi^\pm$ is the restriction of $\varphi$ to $\C_\delta^\pm(x_0,t_0)$.
    \end{enumerate}
    The notions of supersolution and solution are defined as usual.
\end{Definition}

\begin{Remark}
    It is easy to check that in $(ii)$, it is equivalent to take test functions of the form
    \[
    \varphi(x,t)=P(x',t)+p^+x_n^+-p^-x_n^-,
    \]
    where $P$ is a quadratic polynomial in $x'$ and affine in $t$, and $p^\pm \in \R$ satisfy
    \[
    p^+-p^-\geq g(x_0,t_0).
    \]
\end{Remark}

The following lemma will be useful. The proof is similar to \cite[Lemma~4.2]{soria2023regularity}, so we omit it.

\begin{Lemma} \label{lem:equivdef}
    The above definition for subsolutions to \eqref{eq:flatpb} is equivalent to replacing $(ii)$ with the following statement: If $(x_0,t_0)\in \C_1'$ and $\varphi \in C^2 ( \overline{\C_\delta^-}(x_0,t_0))\cap C^2(\overline{\C_\delta^+}(x_0,t_0))$, then either
    $$
    \partial_t \varphi (x_0,t_0) - F^\pm (D^2\varphi(x_0,t_0)) \leq f^\pm (x_0,t_0) \quad \text{or} \quad \varphi_{x_n}^+(x_0,t_0) - \varphi_{x_n}^-(x_0,t_0) \geq g(x_0,t_0).
    $$
\end{Lemma}

We will prove existence and uniqueness of viscosity solutions, and piecewise $C^{1,\alpha}$ regularity up to the flat interface. An important tool will be the $\vep$-envelopes.

\subsection{Upper and lower $\vep$-envelopes}

We consider a family of approximations in the $x'$ and $t$ variables, which play the same role as the Jensen's approximate solutions for elliptic equations (see \cite[Section~5.1]{Caffarelli-Cabre}).

\begin{Definition}\label{Def:jensen}
    Let $0<\rho<1$ and $\vep>0$. Given $u\in \usc(\C_1)$, we define the upper $\vep$-envelope of $u$ on $\overline\C_\rho$, with respect to the $x'$ and $t$ variables, as the function
    $$
u^\vep ((y',y_n),s) = \sup_{(x,t)\in \overline\C_\rho\cap \{x_n=y_n\}} \Big\{ u((x',y_n),t) - \frac{1}{\vep} |x'-y'|^2 - \frac{1}{\vep} (t-s)^2\Big\}
\quad \text{ for } (y,s)\in \overline\C_\rho.
    $$
Similarly, given $u \in \lsc(\C_1)$, we define the lower $\vep$-envelope of $u$ on $\overline\C_\rho$, with respect to the $x'$ and $t$ variables, as the function
    $$
u_\vep ((y',y_n),s) = \inf_{(x,t)\in \overline\C_\rho\cap \{x_n=y_n\}} \Big\{ u((x',y_n),t) + \frac{1}{\vep} |x'-y'|^2 + \frac{1}{\vep} (t-s)^2\Big\}
\quad \text{ for } (y,s)\in \overline\C_\rho.
    $$
\end{Definition}

\begin{Remark}
    Note that since $u\in \usc(\C_1)$, then the supremum in the definition of $u^\vep$ is attained. Hence, for any $(y,s)\in \overline\C_\rho$, there exists $(x_\vep, t_\vep)\in \overline\C_\rho\cap \{x_n=y_n\}$ such that 
    \begin{equation} \label{eq:supatt}
            u^\vep(y,s) = u(x_\vep, t_\vep) - \frac{1}{\vep} |x_\vep'-y'|^2 - \frac{1}{\vep} (t_\vep-s)^2.
    \end{equation}
    Furthermore, if $u$ is bounded in $\C_1$, then the above identity yields
    $$
    |x_\vep'-y'|^2 +  (t_\vep-s)^2 \leq 2\vep \|u\|_{L^\infty(\C_1)}, \quad \text{ for all } (y,s) \in \overline\C_\rho.
    $$
\end{Remark}

\begin{Lemma}\label{Lm:prop_jensen}
    The following properties hold:
    \begin{enumerate}
        \item[$(i)$] $u^\vep \geq u$ on $\overline{\C_\rho}$ and $\limsup_{\vep\to0} u^\vep = u$.
        \item[$(ii)$] $u^\vep$ is Lipschitz continuous with respect to $y'$ and $s$, and
        $$
        |u^\vep(y_1,s_1)-u^\vep(y_2,s_2)| \leq \frac{4\rho}{\vep} \big (|y_1'-y_2'| + |s_1-s_2| \big),
        $$
        for all $(y_1,s_1), (y_2,s_2) \in \overline{\C_\rho}$ with $(y_1)_n=(y_2)_n$.
        \item[$(iii)$] $u^\vep(\cdot, s)$ is punctually second order differentiable in $y'$ a.e.~in $B_\rho$ and for all $s\in [-\rho^2, 0]$.
          \end{enumerate}
\end{Lemma}

\begin{proof}
The proofs of $(i)$ and $(ii)$ are standard.
To prove $(iii)$, by definition of $u^\vep$, for all $s\in [-\rho^2,0]$, 
$$v(\cdot,s)=u^\vep(\cdot,s) + \frac{1}{\vep}|y'|^2$$
is the supremum of affine functions in $y'$. Hence, for all s, $v(\cdot, s)$ is a convex function on $\overline\C_\rho$ with respect to $y'$. By Alexandrov's theorem (see \cite[Theorem~1.5]{Caffarelli-Cabre}), it follows that $u^\vep(\cdot, s)$ is punctually second order differentiable in $y'$ for almost every $y\in B_\rho$ and for all $s\in [-\rho^2,0]$.
\end{proof}

Given a uniformly continuous function $h: \R^{n+1}\to \R$, we define its modulus of continuity with respect to the Euclidean distance as
$$
\omega_h (r) := \sup_{|x-y|^2 + |t-s|^2\leq r^2}\, |h(x,t)-h(y,s)| \quad \text{ for } r>0.
$$

\begin{Proposition}\label{Prop:jensen_solution}
    Let $f^\pm \in C(\C_1^\pm)$ and $g\in C(\C_1')$. If $u$ is a bounded viscosity subsolution to \eqref{eq:flatpb}, then for any $\vep>0$ small and $0<\rho <1$, the upper $\vep$-envelope  $u^\vep$ in $\overline\C_{\rho}$ is a viscosity subsolution to
    $$
    \begin{cases}
       \partial_t u^\vep -  F^\pm (D^2 u^\vep) = f_\vep^\pm & \text{ in } \C_r^\pm,\\
        (u^\vep)_{x_n}^+ - (u^\vep)_{x_n}^- = g_\vep & \text{ on } \C_r',
    \end{cases}
    $$
    with $r= \rho- r_\vep$, where $r_\vep = (2\vep \|u\|_{L^\infty(\C_1)})^\frac{1}{2}$,  $f_\vep^\pm = f^\pm+ \omega_{f^\pm}(r_\vep)$, and $g_\vep = g - \omega_g (r_\vep)$. A similar result holds for the lower $\varepsilon$-envelope $u_\varepsilon$ when $u$ is a supersolution.
\end{Proposition}

\begin{proof}
    Let $\varphi\in C^2(\C_r^+)$ touch $u^\varepsilon$ from above at a point $(\bar x, \bar t)$ in $\C_\delta(\bar x, \bar t)\subset \C_r^+$, for some $\delta>0$. Then, for $\varepsilon$ small, by \eqref{eq:supatt}, there exists $(\bar x_\varepsilon, \bar t_\varepsilon)\in \C_\rho\cap \{x_n=\bar x_n\}$ such that
    \[
    u^\vep(\bar x,\bar t) = u(\bar x_\vep, \bar t_\vep) - \tfrac{1}{\vep} d_\vep,
    \]
    where $d_\vep:= |\bar x_\varepsilon'-\bar x'|^2+(\bar t_\varepsilon-\bar t)^2\leq 2\varepsilon\|u\|_{L^\infty(\C_1)}.$
    Consider
    \[
    \phi(y,s) :=\varphi(y+\bar x-\bar x_\varepsilon,s+\bar t-\bar t_\varepsilon)+\tfrac{1}{\varepsilon}d_\varepsilon.
    \]
    Choosing $(y,s)$ sufficiently close to $(\bar x_\varepsilon,\bar t_\varepsilon)$ and $s\leq\bar t_\varepsilon$, we have $(y+\bar x-\bar x_\varepsilon,s+\bar t-\bar t_\varepsilon)\in \C_\delta(\bar x,\bar t)$. By the definition of $u^\varepsilon$,
    \begin{align*}
        u(y,s)\leq &\, u^\varepsilon(y+\bar x-\bar x_\varepsilon,s+\bar t-\bar t_\varepsilon)+\tfrac{1}{\varepsilon}d_\varepsilon
        \leq \varphi(y+\bar x-\bar x_\varepsilon,s+\bar t-\bar t_\varepsilon)+\tfrac{1}{\varepsilon}d_\varepsilon=\phi(y,s),
    \end{align*}
    with equality at $(y,s)=(\bar x_\varepsilon,\bar t_\varepsilon)$. Thus, $\phi$ touches $u$ from above at $(\bar x_\varepsilon, \bar t_\varepsilon)$, and therefore,
    \begin{align*}
        \p_t\varphi(\bar x,\bar t)-F^+(D^2\varphi(\bar x, \bar t)) &=\p_t \phi(\bar x_\varepsilon,\bar t_\varepsilon)-F^+(D^2\phi(\bar x_\varepsilon,\bar t_\varepsilon))
        \leq f^+(\bar x_\varepsilon,\bar t_\varepsilon)\leq f^+(\bar x,\bar t)+\omega_{f^+}(r_\vep),
    \end{align*}
    where $r_\vep = (2\vep \|u\|_{L^\infty(\C_1)})^\frac{1}{2}$. The proof for $(\bar x, \bar t)\in \C_r^-$ is analogous.
    
To check the transmission condition, let $(\bar x, \bar t)=(\bar x',0,\bar t)\in \C_r'$, and assume that
\[
\varphi(y,s)=P(y',s)+p^+y_n^+-p^-y_n^-
\]
touches $u^\varepsilon$ from above at $(\bar x,\bar t)$, where $P$ is a quadratic polynomial in $y'$ and affine in $s$. Then
\[
\phi(y,s) :=P(y'+\bar x'-\bar x_\varepsilon,s+\bar t-\bar t_\varepsilon)+\tfrac{1}{\varepsilon}d_\varepsilon+p^+y_n-p^-y_n^-
\]
touches $u$ from above at $(\bar x_\varepsilon,\bar t_\varepsilon)$. Therefore,
\begin{align*}
    p^+-p^-=(\phi^+_{x_n}-\phi^-_{x_n})(\bar x_\varepsilon,\bar t_\varepsilon)
    \geq g(\bar x_\varepsilon,\bar t_\varepsilon)\geq g(\bar x,\bar t)-\omega_g(r_\vep).
\end{align*}
\end{proof}

\subsection{Half-relaxed limits}

We recall the classical notion of half-relaxed limits and some of its properties that will be useful for future proofs. For more details, see \cite{Crandall-Ishii-Lions}.

\begin{Definition}
    Let $\{u_k\}_{k=1}^\infty$ be a sequence of functions defined on $\overline\C_1$. We define the upper relaxed-limit as
    $$
    {\limsup}^* \, u_k (x,t) := \limsup_{\substack{(y,s)\to (x,t)\\ k\to \infty} } u_k(y,s) \quad \text{ for } (x,t) \in \overline\C_1.
    $$
    Similarly, we define the lower relaxed-limit as
     $$
    {\liminf}_* \, u_k(x,t) := \liminf_{\substack{(y,s)\to (x,t)\\ k\to \infty} } u_k(y,s)  \quad \text{ for } (x,t) \in \overline\C_1.
    $$
\end{Definition}
\begin{Remark}
    It is well-known that ${\limsup}^* \, u_k  \in \usc(\overline\C_1)$ and ${\liminf}_* \, u_k \in \lsc (\overline\C_1)$.
\end{Remark}

The following lemma is the parabolic counterpart of \cite[Proposition~4.3]{Crandall-Ishii-Lions}.
\begin{Lemma} \label{Lm:relaxedlimit}
    Let $\{u_k\}_{k=1}^\infty \subseteq \usc(\overline\C_1)$ and let $u={\limsup}^* \, u_k$.  Assume that a continuous function $\varphi$ touches $u$ from above at $(x_0,t_0) \in \overline\C_1$. 
    Let $\vep >0$.
    Then there are indexes $k_j\to \infty$, points $(x_j,t_j)\to (x_0,t_0)$ in $\overline\C_1$, and continuous functions,
      $$
    \varphi_j(x,t)= \varphi(x,t) - \varphi(x_j,t_j) + u_{k_j}(x_j,t_j) + \vep ( |x-x_0|^2 - |x_j-x_0|^2 + t_j-t ),
    $$
    such that $\varphi_j$ touches $u_{k_j}$ from above at $(x_j,t_j)$, and
    $$
    u(x_0,t_0)=\lim_{j\to \infty} u_{k_j}(x_j,t_j).
    $$
\end{Lemma}

\subsection{Comparison Principle and Uniqueness}

The main theorem in this section is the following.

\begin{Theorem}\label{thm:difference}
    Let $f_1, f_2 \in C(\C_1\setminus \C_1') \cap L^\infty(\C_1)$, and let $g_1,g_2 \in C(\C_1')$. Assume that $u \in \usc(\overline\C_1)$ and $v\in \lsc (\overline\C_1)$ are bounded, and satisfy in the viscosity sense,
    $$
    \begin{cases}
        u_t - F^\pm (D^2 u) \leq f_1^\pm & \text{ in } \C_1^\pm,\\
        u_{x_n}^+ - u_{x_n}^- \geq g_1 & \text{ on } \C_1',
    \end{cases}
\quad \text{ and } \quad
    \begin{cases}
        v_t - F^\pm (D^2 v) \geq f_2^\pm & \text{ in } \C_1^\pm,\\
        v_{x_n}^+ - v_{x_n}^- \leq g_2 & \text{ on } \C_1'.
    \end{cases}
    $$
Then $w=u-v$ satisfies in the viscosity sense,
    $$
    \begin{cases}
            w_t - \Mp(D^2w) \leq f_1^\pm - f_2^\pm & \text{ in } \C_1^\pm,\\
            w_{x_n}^+ - w_{x_n}^- \geq g_1-g_2 & \text{ on } \C_1'.
    \end{cases}
    $$
\end{Theorem}

\begin{proof}
    By classical theory, $w$ is a viscosity subsolution to the equation in the interior of $\C_1^\pm$, and thus, we only need to check that the transmission condition is satisfied in the viscosity sense.
    
    Let $(\bar x,\bar t)=(\bar x', 0, \bar t)\in \C_1'$, and assume that $P(x',t)+p^+x_n^+-p^-x_n^-$ touches $w$ by above at $(\bar x,\bar t)$, where $P$ is a quadratic polynomial in $x'$ and affine in $t$, and $p^\pm\in \R$. We need to show that 
    \begin{align}\label{eq:conclusion_comparison}
        p^+-p^-\geq g_1(\bar x,\bar t)-g_2(\bar x,\bar t).
    \end{align}
Fix $\tau>0$ and $C>0$ to be chosen large. Then
\[
\varphi(x,t)=P(x',t)+(p^++\tau)x_n^+-(p^--\tau)x_n^--Cx_n^2
\]
touches $w$ strictly by above at $(\bar x,\bar t)$, possibly in a smaller neighborhood where $\tau |x_n|-Cx_n^2>0$.

For $\varepsilon>0$, consider $w_\varepsilon=u^\varepsilon-v_\varepsilon$, where $u^\varepsilon$ and $v_\varepsilon$ are given in Definition~\ref{Def:jensen}. By Lemma~\ref{Lm:prop_jensen}~$(i)$, we have $\limsup_{\vep\to0} w_\varepsilon= w$. By Lemma~\ref{Lm:relaxedlimit}, up to a subsequence, there exist points $(x_\varepsilon,t_\varepsilon)\to (\bar x,\bar t)$ and continuous functions 
\[
\varphi_\varepsilon(x,t)=\varphi(x,t)-\varphi(x_\varepsilon,t_\varepsilon)+w_\varepsilon(x_\varepsilon,t_\varepsilon)+|x-\bar x|^2-|x_\varepsilon-\bar x|^2+t_\varepsilon-t,
\]
which touch $w_\varepsilon$ strictly from above at $(x_\varepsilon,t_\varepsilon)$ in $\C_\delta(x_\varepsilon,t_\varepsilon)$, with $\delta>0$ small. In particular, there exists $\eta>0$ such that $\varphi_\varepsilon-w_\varepsilon\geq \eta>0$ on $\p_p\C_\delta(x_\varepsilon,t_\varepsilon)$. By Proposition \ref{Prop:jensen_solution},
\begin{align}\label{eq:jen_w}
    \p_tw_\varepsilon-\Mp(D^2w_\varepsilon)\leq (f_1^\pm)_\varepsilon-(f_2^\pm)_\varepsilon \quad\text{ in } \C_\rho^\pm,
\end{align}
in the viscosity sense, for some $0<\rho<1$ such that $\overline\C_\delta(x_\varepsilon,t_\varepsilon)\subset \C_\rho$. Take $C>0$ large so that
\begin{align}\nonumber
        \p_t\varphi_\varepsilon-\Mp(D^2\varphi_\varepsilon)&\geq -1-|\p_tP|-\Lambda|D^2_{x'}P|-2\Lambda+2\lambda(C-1)\\ \label{eq:vareps}
    &\geq \sup_{\C_\rho^\pm}\left\{(f_1^\pm)_\varepsilon-(f_2^\pm)_\varepsilon\right\} \ \text{ in } \C_\rho^\pm.
\end{align}
Since $\varphi_\varepsilon$ touches $w_\varepsilon$ by above at $(x_\varepsilon,t_\varepsilon)$, by \eqref{eq:jen_w} and \eqref{eq:vareps}, we must have $(x_\varepsilon)_n=0$.

Define the function
\begin{align}\label{eq:psi1}
    \psi:=\varphi_\varepsilon-w_\varepsilon-\eta/2.
\end{align}
Then $\psi\geq \eta/2>0$ on $\p_p\C_\delta(x_\varepsilon,t_\varepsilon)$ and $\psi(x_\varepsilon,t_\varepsilon)=-\eta/2<0$ is a  minimum value. Let $C_\psi$ be the convex envelope of $-\psi_-$ in $\C_{2\delta}'(x_\varepsilon,t_\varepsilon)$, where as usual we have extended $-\psi_-\equiv 0$ outside $\overline{\C_\delta'}(x_\varepsilon,t_\varepsilon)$. By Lemma~\ref{Lm:prop_jensen}~$(iii)$, we have $\psi\in C^{1,1}_{x',t}$ by above in $\C_\rho$. Hence, for any $(x_0',t_0)\in \overline{\C_\delta'}(x_\varepsilon,t_\varepsilon)$, there exists a convex paraboloid, with respect to both $x'$ and $t$, that touches $\psi(x',0,t)$ by above at $(x_0',t_0)$. Moreover, $C_\psi\in C^{1,1}(\overline{\C_\delta'}(x_\varepsilon,t_\varepsilon))$ in the parabolic sense.

For $\gamma>0$, let
\[
D_\gamma := \big\{(x',t)\in \overline{\C_\delta'}(x_\varepsilon,t_\varepsilon)\,:\,C_\psi(x',t)=\psi(x',0,t)\text{ and } |\nabla'C_\psi(x',t)|\leq \gamma \big\}.
\]
Since $C_\psi(x_\varepsilon',t_\varepsilon)=\psi(x_\varepsilon',0,t_\varepsilon)$ and $\nabla'C_\psi(x_\varepsilon',t_\varepsilon)=0$, we have $(x_\varepsilon',t_\varepsilon)\in D_\gamma\neq  \emptyset$, for any $\gamma>0$.  
We also know by Remark~\ref{Rmk:ABP_flat} and $\psi_-\not\equiv 0$ that the contact set (without the gradient condition) has positive measure. Since $\nabla'C_\psi$ is continuous, it follows that $|D_\gamma|>0$, for any $\gamma>0$.

Choose $0<\gamma\leq \frac{\eta}{4\delta}$. Then there exists $(y_\varepsilon, t_\varepsilon)\in D_\gamma$ such that both $u^\varepsilon$ and $v_\varepsilon$ are punctually second order differentiable w.r.t.~$x'$ at $(y_\varepsilon,t_\varepsilon)$. This follows from the fact that $u^\varepsilon$ and $-v_\varepsilon$ are semi-convex in $x'$, for every $t$. Furthermore, we have that
\[
\ell(x')=\nabla' C_\psi (y_\varepsilon',t_\varepsilon)\cdot (x'-y_\varepsilon')+\psi(y_\varepsilon,t_\varepsilon)
\]
touches $\psi$ from below at $(y_\varepsilon, t_\varepsilon)$ on $\C_\delta(x_\varepsilon,t_\varepsilon)$. Indeed, by the second definition of the convex envelope in  \eqref{Def:convex_env},
\[
\ell(x')\leq \psi(x',0,t) \quad \text{ for all } (x',t)\in \C'_\delta(x_\varepsilon,t_\varepsilon)
\]
and
\begin{align*}
    \ell(x')\leq|\nabla'C_\psi(y_\varepsilon,t_\varepsilon)|\,|x'-y_\varepsilon'|\leq \frac{\eta}{4\delta}(2\delta)=\frac{\eta}{2}\leq \psi(x,t) \quad\text{ on } \p_p\C_\delta(x_\varepsilon,t_\varepsilon).
\end{align*}
Therefore, $\ell\leq \psi$ on $\p_p\C^\pm_\delta(x_\varepsilon,t_\varepsilon)$. In particular, by \eqref{eq:psi1}, $w_\varepsilon\leq \varphi_\varepsilon-\ell-\eta/2$ on $\p_p\C^\pm_\delta(x_\varepsilon,t_\varepsilon)$. Furthermore, by \eqref{eq:jen_w}, \eqref{eq:vareps}, and the comparison principle, we get
\[
w_\varepsilon\leq \varphi_\varepsilon-\ell-\eta/2 \quad \text{ on } \overline{\C}_\delta(x_\varepsilon,t_\varepsilon).
\]
Define $\bar \varphi :=\varphi_\varepsilon-\ell-\eta/2$. Consider the viscosity solutions $\bar u^\varepsilon$ and $\bar v_\varepsilon$ of the Cauchy-Dirichlet problems
\begin{align*}
    \begin{cases}
        \p_t\bar u^\varepsilon-F^\pm(D^2\bar u^\varepsilon)=(f_1^\pm)_\varepsilon & \text{ in } \C^\pm_\delta(x_\varepsilon,t_\varepsilon),\\
        \bar u^\varepsilon=u^\varepsilon & \text{ on } \p_p\C^\pm_\delta(x_\varepsilon,t_\varepsilon),
    \end{cases}
\end{align*}
and
\begin{align*}
    \begin{cases}
        \p_t\bar v_\varepsilon-F^\pm(D^2\bar v_\varepsilon)=(f_2^\pm)_\varepsilon & \text{ in } \C^\pm_\delta(x_\varepsilon,t_\varepsilon),\\
        \bar v_\varepsilon=v_\varepsilon & \text{ on } \p_p\C^\pm_\delta(x_\varepsilon,t_\varepsilon),
    \end{cases}
\end{align*}
where $(f_i^\pm)_\varepsilon$, $i=1,2$, are defined as in Proposition \ref{Prop:jensen_solution}. By the comparison principle for the Cauchy-Dirichlet problem, we have $\bar u^\varepsilon\geq u^\varepsilon$ and $\bar v_\varepsilon\leq v_\varepsilon$ in $\C_\delta(x_\varepsilon,t_\varepsilon)$, and thus,
\begin{equation}\label{eq:grad_jen}
    \begin{aligned}
        (\bar u^\varepsilon)_{x_n}^+-(\bar u^\varepsilon)_{x_n}^-\geq (g_1)_\varepsilon \quad \text{ and } \quad
        (\bar v_\varepsilon)_{x_n}^+-(\bar v_\varepsilon)_{x_n}^-\leq (g_2)_\varepsilon
        \quad \text{ on } \C'_\delta(x_\varepsilon,t_\varepsilon),
    \end{aligned}
\end{equation}
in the viscosity sense, where $(g_i)_\varepsilon$, $i=1,2$, are defined as in Proposition \ref{Prop:jensen_solution}. Recall that $u^\varepsilon,v_\varepsilon\in C^{1,\alpha}(y_\varepsilon,t_\varepsilon)$ with respect to both $x'$ and $t$. Hence, by pointwise parabolic $C^{1,\alpha}$ estimates, there are linear polynomials $\ell_u^\pm$ and $\ell_v^\pm$ such that for all $r$ small,
\[
\|(\bar u^\varepsilon)^\pm-\ell^\pm_u\|_{L^\infty(\C_r^\pm(y_\varepsilon,t_\varepsilon))} + \|(\bar v_\varepsilon)^\pm-\ell^\pm_v\|_{L^\infty(\C_r^\pm(y_\varepsilon,t_\varepsilon))}\leq C_\varepsilon r^{1+\alpha}.
\]
Denote by $p_u^\pm=\p_{x_n}\ell_u^\pm$ and $p_v^\pm=\p_{x_n}\ell_v^\pm$. By Lemma~\ref{Lm:TC_pointwise} below (rescaled), we conclude that \eqref{eq:grad_jen} holds pointwise at $(y_\varepsilon, t_\varepsilon)$. Namely, 
\begin{align}\label{eq:pointwise_TC}
    p_u^+-p_u^-\geq (g_1)_\varepsilon(y_\varepsilon,t_\varepsilon) \quad \text{ and } \quad p_v^+-p_v^-\leq (g_2)_\varepsilon(y_\varepsilon,t_\varepsilon).
\end{align}
Let $\bar w_\varepsilon=\bar u^\varepsilon-\bar v_\varepsilon$. By the comparison principle, 
\[
\bar \varphi\geq \bar w_\varepsilon \ \text{ in } \C_\delta(x_\varepsilon,t_\varepsilon)\quad \text{ and }\quad \bar \varphi(y_\varepsilon,t_\varepsilon)=\bar w(y_\varepsilon,t_\varepsilon).
\]
Since $\bar w_\varepsilon\in C^{1,\alpha}(y_\varepsilon,t_\varepsilon)$, we have that 
\begin{align*}
    p^++\tau =\bar \varphi_{x_n}^+(y_\varepsilon,t_\varepsilon)\geq (\bar w_\varepsilon)_{x_n}^+(y_\varepsilon,t_\varepsilon)=p_u^+-p_v^+,\\
    p^--\tau =\bar \varphi_{x_n}(y_\varepsilon,t_\varepsilon)\leq (\bar w_\varepsilon)_{x_n}^-(y_\varepsilon,t_\varepsilon)=p_u^--p_v^-.
\end{align*}
Therefore, combining the previous inequalities with \eqref{eq:pointwise_TC}, we get
\begin{align*}
    p^+-p^-+2\tau & \geq  (g_1)_\varepsilon(y_\varepsilon, t_\varepsilon)-(g_2)_\varepsilon(y_\varepsilon, t_\varepsilon)\\
    &=(g_1-g_2)(y_\varepsilon, t_\varepsilon)+\omega_{g_1}\big((2\varepsilon\|u\|_\infty)^\frac{1}{2}\big)-\omega_{g_2}\big((2\varepsilon\|v\|_\infty)^\frac{1}{2}\big).
\end{align*}
Recall that $(y_\varepsilon,t_\vep)\in \C_\delta(x_\varepsilon, t_\vep)$,  and $x_\varepsilon\to \bar x$, $t_\varepsilon\to \bar t$ as $\varepsilon\to 0$. Hence, letting first $\tau \to 0$, then $\delta \to 0$ so that $(y_\varepsilon,t_\varepsilon)\to (x_\varepsilon,t_\varepsilon)$, and finally $\varepsilon\to 0$, we conclude \eqref{eq:conclusion_comparison}. 
\end{proof}

\begin{Lemma}\label{Lm:TC_pointwise}
    Let $f\in C(\C_1\setminus \C_1')\cap L^\infty(\C_1)$ and $u$ be a viscosity solution to 
    \begin{align*}
        \begin{cases}
            \p_tu-F^\pm(D^2u)=f^\pm &\text{ in } \C_1^\pm,\\
            u_{x_n}^+-u_{x_n}^-\geq 0  & \text{ on } \C_1'.
        \end{cases}
    \end{align*}
    If $u$ is twice differentiable at $0$ in the $x'$-direction and Lipschitz in the $t$-direction, then $u$ is differentiable w.r.t. $x$ at the origin and satisfies the transmission condition in the classical sense:
    \[
    u_{x_n}^+(0)-u_{x_n}^-(0)\geq 0.
    \]
\end{Lemma}

\begin{proof}
The proof is an adaptation of \cite[Lemma~4.3]{DFS2018}, so we only highlight the differences. 
    From pointwise $C^{1,\alpha}$ estimates at the boundary in \cite[Theorem 3.1]{Ma-Moreira-Wang}, there are affine functions $\ell^\pm(x)$ such that, for all $r$ small,
    \[
    \|u-\ell^\pm\|_{L^\infty(\C_r^\pm)}\leq Cr^{1+\alpha}.
    \]
    Up to subtracting an affine function, we can assume that $\ell^\pm=u_{x_n}^\pm (0) x_n=:d^\pm x_n$.

    Let $w$ be the solution to
    \begin{align*}
        \begin{cases}
            \p_t w-F^+(D^2w)=f^+ &\text{ in } \C_r^+,\\
            w=\varphi_r &\text{ on } \p_p \C_r,
        \end{cases}
    \end{align*}
    where $\varphi_r$ is defined by
    \begin{align*}
        \varphi_r(x,t):=\begin{cases}
            2C(|x|^{1+\alpha}-r^{\alpha-1}t) &\text{ on } \p_L\C_r^+\cap \{x_n>0\},\\
            2Cr^{\alpha-1}(|x'|^2-t) &\text{ on } \C_r',\\
            \psi_r &\text{ on } \p_B\C_r^+,
        \end{cases}
    \end{align*}
    and $\psi_r(x)$ is defined as the solution to 
    \begin{align*}
        \begin{cases}
            \Delta \psi_r=0 & \text{ in } B_r^+,\\
            \psi_r= 2C(|x|^{1+\alpha}+r^{1+\alpha}) & \text{ on } \p B_r^+\cap \{x_n>0\},\\
            \psi_r=2C(r^{\alpha-1}|x'|^2+r^{1+\alpha}) &\text{ on } B_r'.
        \end{cases}
    \end{align*}
    
Arguing similarly to \cite[Lemma~4.3]{DFS2018}, we obtain
\begin{align*}
    w(x,t)\leq 2Cr^{\alpha-1}(|x'|^2-t)+Cr^\alpha x_n \quad \text{ on } \overline{\C_{r/2}^+}.
\end{align*}
From the comparison principle, we see that
\[
\varphi(x,t) :=2Cr^{\alpha-1}(|x'|^2-t)+(\bar C r^\alpha+d^+) x_n^+-(-\bar C r^\alpha+d^-)x_n^-
\]
touches $u$ from above at $0$. Therefore,
\[
(\bar Cr^\alpha+d^+)-(-\bar Cr^\alpha+d^-)\geq 0
\]
holds for all small $r>0$, and thus,
\[
d^+-d^-=  u_{x_n}^+(0)-u_{x_n}^-(0)\geq 0.
\] 
\end{proof}

The comparison principle and uniqueness are straightforward consequences of Theorem~\ref{thm:difference}.

\begin{Theorem}[Comparison Principle] \label{thm:comparison}
Assume that $u \in \usc(\overline\C_1)$ is a bounded viscosity subsolution and $v\in \lsc(\overline\C_1)$ is a bounded viscosity supersolution to \eqref{eq:flatpb}. If $u\leq v$ on $\p_p \C_1$, then
$$
u \leq v \quad \text{ in } \C_1.
 $$
\end{Theorem}

\begin{proof}
    Let $w=v-u$. By Theorem~\ref{thm:difference}, we have
    $$
      \begin{cases}
            \partial_t w - \Mp(D^2w) \geq 0 & \text{ in } \C_1^\pm,\\
            w_{x_n}^+ - w_{x_n}^- \leq 0 & \text{ on } \C_1'.
    \end{cases}
    $$
    By the maximum principle (Corollary~\ref{cor:MP}), it follows that $w\geq 0$ in $\C_1$.
\end{proof}

\begin{Corollary}[Uniqueness]\label{Cor:Uniqueness}
    The transmission problem
    \[
    \begin{cases}
    \partial_t u - F^\pm(D^2u)=f^\pm & \text{ in } \C_1^\pm,\\
    u_{x_n}^+-u_{x_n}^- = g & \text{ on } \C_1',\\
    u=\varphi  & \text{ on } \p_p\C_1
    \end{cases}
    \]
    has at most one solution.
\end{Corollary}

\subsection{Existence}

We prove existence of viscosity solutions to flat interface problems via Perron's method. 

\begin{Theorem}[Existence] \label{thm:existence}
    Let $f \in C(\C_1 \setminus \C_1')\cap L^\infty(\C_1)$, $g\in C(\C_1')$, and $\phi \in C (\p_p \C_1)$. Then there exists a unique viscosity solution $u\in C(\overline\C_1)$ to \eqref{eq:flatpb} such that $u=\phi$ on $\p_p\C_1$.
\end{Theorem}

First, we need to construct suitable barriers.

\begin{Lemma}[Perron barriers]\label{lem:barriers}
    Let $\phi \in C(\p_p \C_1)$. There exist functions $\uu, \bu \in C^2(\C_1\setminus \C_1')\cap C(\overline\C_1)$ such that $\uu$ is a viscosity subsolution, $\bu$ is a viscosity supersolution to \eqref{eq:flatpb}, and such that
    $$
\begin{cases}
    \uu \leq \bu & \text{ in } \C_1,\\
    \uu = \bu = \phi & \text{ on } \p_p \C_1.
\end{cases}
    $$
\end{Lemma}

\begin{proof}
   Let $\underline\psi, \overline \psi \in C^2(\C_1)\cap C(\overline\C_1)$ be the unique solutions to 
    $$
\begin{cases}
    \p_t \underline \psi -\Mm(D^2\underline \psi) = \|f\|_{L^\infty(\C_1)} & \text { in } \C_1,\\
    \underline \psi = \phi - \tfrac{1}{2}\|g\|_{L^\infty(\C_1')}|x_n| & \text{ on } \p_p \C_1,
\end{cases}
$$
and
$$
\begin{cases}
    \p_t \overline \psi -\Mp(D^2\overline \psi) = -\|f\|_{L^\infty(\C_1)} & \text { in } \C_1,\\
    \overline \psi = \phi + \tfrac{1}{2}\|g\|_{L^\infty(\C_1')}|x_n| & \text{ on } \p_p \C_1.
\end{cases}
    $$
    Define the functions $\uu, \bu \in C^2(\C_1\setminus \C_1')\cap C(\overline\C_1)$ as
    $$
    \uu = \underline \psi + \tfrac{1}{2}\|g\|_{L^\infty(\C_1')}|x_n|
    \quad \text{ and } \quad
       \bu = \overline \psi - \tfrac{1}{2}\|g\|_{L^\infty(\C_1')}|x_n|.
    $$
    Then $\uu=\bu=\phi$ on $\p_p \C_1$. By construction, $\uu$ is a subsolution and $\bu$ is a supersolution to \eqref{eq:flatpb}. By the comparison principle (Theorem~\ref{thm:comparison}), we have $\uu\leq \bu$ in $\C_1$.
\end{proof}

We define the set of admissible subsolutions as
$$
\mathcal{A} := \big\{ v \in \usc(\overline\C_1) : \underline u \leq v \leq \overline u \text{ and $v$ is a viscosity subsolution to \eqref{eq:flatpb}} \big\},
$$
where $\underline u$ and $\overline u$ are given in Lemma~\ref{lem:barriers}. Note that $\underline u \in \mathcal{A}$, so $\mathcal{A}\neq \emptyset$. Set
\begin{equation} \label{eq:supsub}
u(x,t) := \sup \big \{v(x,t) : v \in \A \big\}.
\end{equation}
Let $u_*\in \lsc(\overline\C_1)$ and $u^* \in \usc(\overline\C_1)$ denote the lower and upper semicontinuous envelopes of $u$ on $\overline{\C_1}$, respectively. By definition,
$
u_*\leq u \leq u^*.
$

\begin{Lemma} \label{lem:HRLsub}
    If $\{v_k\}_{k=1}^\infty\subset \A$, then $v :=\limsup^* v_k \in \A$.
\end{Lemma}

\begin{proof}
    Let $\{v_k\}_{k=1}^\infty\subset \A$ and set $v :=\limsup^* v_k$. 
    Since $\uu \leq v_k \leq \bu$, for all $k\geq1,$ then the same holds for $v$.
    We need to show that $v$ is a viscosity subsolution to \eqref{eq:flatpb}. Indeed, let $(x_0,t_0) \in \C_1$, and assume that $\varphi$ is a test function touching $v$ from above at $(x_0,t_0)$. Let $\vep>0$. By Lemma~\ref{Lm:relaxedlimit}, there are indexes $k_j\to \infty$, points $(x_j,t_j)\to (x_0,t_0)$ in $\overline{\C_1}$, and test functions,
    $$
      \varphi_j(x,t)= \varphi(x,t) - \varphi(x_j,t_j) + u_{k_j}(x_j,t_j) + \vep( |x-x_0|^2 - |x_j-x_0|^2 + t_j-t),
      $$
    such that $\varphi_j$ touches $v_{k_j}$ from above at $(x_j,t_j)$, and
    $$
    v(x_0,t_0)=\lim_{j\to \infty} u_{k_j}(x_j,t_j).
    $$
    If $(x_0,t_0) \in \C_1^\pm$, then for $j$ sufficiently large, we may assume that $(x_j,t_j)\in \C_1^\pm$. Since $v_{k_j}\in \A$, and $\varphi_j$ touches $v_{k_j}$ from above at $(x_j, t_j)$, then 
  $$
     \p_t \varphi (x_j, t_j) - \vep - F^\pm( D^2\varphi(x_j,t_j) + 2\vep I) \leq f^\pm(x_j,t_j).
  $$
  Letting $j\to \infty$ and $\vep \to 0$, by continuity of $F^\pm$, $\p_t \varphi$, $D^2\varphi$, and $f^\pm$, we get
   $$\partial_t \varphi(x_0,t_0) - F^\pm (D^2(x_0,t_0))\leq f^\pm(x_0,t_0).$$

  If $(x_0,t_0)\in \C_1'$, then either there exists $j_0\geq 1$ such that for all $j\geq j_0$ we have $(x_j,t_j)\in \C_1^\pm$, or for all $j_0\geq 1$, there exists $j\geq j_0$ such that $(x_j,t_j)\in \C_1'$. In the first case, we get the above inequality, following the previous argument. In the second case, we have 
  $$
  \partial_{x_n}\varphi_j^+(x_j,t_j) - \partial_{x_n}\varphi_j^-(x_j,t_j)\geq g_j(x_j,t_j).
  $$
  Letting $j\to \infty$, we get the desired estimate. 

By Lemma~\ref{lem:equivdef}, it follows that $v$ is a viscosity subsolution to \eqref{eq:flatpb} in $\C_1$.
\end{proof}

We divide the proof of Theorem~\ref{thm:existence} into two steps.

\begin{Lemma}[Step 1] \label{lem:step1}
    $u^*$ is a viscosity subsolution to \eqref{eq:flatpb}. 
\end{Lemma}

\begin{proof}
    Let $(x_0,t_0)\in \C_1$. Assume that $\varphi$ is a test function touching $u^*$ from above at $(x_0,t_0)$. By \eqref{eq:supsub} and the definition of $u^*$, there exist points $(x_k,t_k) \in \C_1$ and functions $\{v_k\}_{k=1}^\infty\subset \A$ such that $(x_k,t_k)\to (x_0, t_0)$ and
    $$
    u^*(x_0,t_0) = \limsup_{k\to \infty} v_k(x_k,t_k) \leq {\limsup}^*\, v_k (x_0,t_0).
    $$
    Furthermore, for any $(x,t) \in \C_1$ and $\{v_j\}_{j=1}^\infty \subset \A$, it holds that
    $$
    u^*(x,t) \geq \Big( \sup_{j\geq 1} v_j\Big)^*(x,t)\geq  {\limsup}^*\, v_j(x,t).
    $$
    Hence, $\varphi$ touches $\limsup^* v_k$ from above that $(x_0,t_0)$. By Lemma~\ref{lem:HRLsub}, it follows that $u^*$ is a subsolution to \eqref{eq:flatpb}.
\end{proof}

\begin{Lemma}[Step 2] \label{lem:step2}
    $u_*$ is a viscosity supersolution to \eqref{eq:flatpb}.
\end{Lemma}

\begin{proof}
    Assume, for the sake of contradiction, that $u_*$ is not a viscosity supersolution to \eqref{eq:flatpb}. Then there are $(x_0,t_0)\in \C_1$ and a test function $\varphi$ that touches $u_*$ (strictly) from below at $(x_0,t_0)$,~but 
    \begin{enumerate}
        \item[$(i)$] $\partial_t \varphi - F^\pm (D^2 \varphi(x_0,t_0))< f^\pm(x_0,t_0)$ if $(x_0,t_0) \in \C_1^\pm$;
        \item[$(ii)$] $\varphi_{x_n}^+(x_0,t_0) - \varphi_{x_n}^-(x_0,t_0) > g(x_0,t_0)$ if $(x_0,t_0)\in \C_1'$.
    \end{enumerate}
Hence, by continuity, the function 
$$
\varphi_{\delta,\gamma}(x,t) := \varphi(x,t)+\delta - \gamma \big( |x-x_0|^2 + t_0-t\big)
$$
is a classical strict subsolution in $\C_r(x_0,t_0)$ for sufficiently small $r, \delta, \gamma>0$. 
Moreover, we have $\varphi_{\delta,\gamma}(x_0,t_0) > u_*(x_0,t_0)$ and $\varphi_{\delta,\gamma} < u_*\leq u$ on $\partial_p \C_r(x_0,t_0)$, choosing $\delta\leq \gamma r^2/2$. Consider
$$
v:= \begin{cases}
\max \{u, \varphi_{\delta,\gamma}\} & \text{ in } \C_r(x_0,t_0),\\
u & \text{ otherwise }.
\end{cases}
$$
Then $v \in C(\overline\C_1)$ is a viscosity subsolution to \eqref{eq:flatpb} and $\uu \leq u \leq v$ in $\C_1$.
Moreover, note that $u_*(x_0,t_0) < \bu (x_0,t_0)$. Otherwise, $\varphi$ touches $\bu$ from below at $(x_0,t_0)$, and since $\bu$ is a viscosity supersolution to \eqref{eq:flatpb}, then $(i)$ and $(ii)$ cannot hold. Hence, $\varphi_{\delta,\gamma}(x_0,t_0)=u_*(x_0,t_0)+\delta < \bu(x_0,t_0)$ for $\delta$ small,
and by continuity, $v\leq \bu$ in $\C_1$, taking $r$ smaller if necessary. Therefore, $v\in \A$.

On the other hand, there are points $(x_k,t_k)\to (x_0,t_0)$ such that $u(x_k,t_k) \to u_*(x_0,t_0)$. Hence,
$$
\lim_{k\to \infty} ( v(x_k,t_k) - u(x_k,t_k) ) \geq  \varphi_{\delta,\gamma}(x_0, t_0) - u_*(x_0,t_0)= \delta>0,
$$
which contradicts that $v \leq u$ since $v\in \A$.

\end{proof}

\begin{customproof}{Theorem \ref{thm:existence}}
    By Lemmas~\ref{lem:step1} and \ref{lem:step2}, we have $u^*$ is a viscosity subsolution, and $u_*$ is a viscosity supersolution to \eqref{eq:flatpb} in $\C_1$. Moreover, $u_*=u^*$ on $\partial_p \C_1$. By the comparison principle (Theorem~\ref{thm:comparison}), it follows that $u^*\leq u_*$ in $\C_1$. Therefore, $u_*=u=u^*$, and $u \in C(\overline\C_1)$ is a viscosity solution to \eqref{eq:flatpb}. The uniqueness follows from Corollary~\ref{Cor:Uniqueness}.
\end{customproof}

\subsection{$\Coa$ estimates up to the flat interface}

In future proofs, we will need the following $C^{1,\alpha}$ estimate for the homogeneous problem. 

\begin{Proposition}\label{prop:regflat}.
Suppose that $v$ is a bounded viscosity solution to
\[
\begin{cases}
\p_t v-F^\pm(D^2 v) = 0 & \text{in } \C_1^\pm, \\
v_{x_n}^+ - v_{x_n}^- = 0 & \text{on } \C_1'.
\end{cases}
\]
Then $v \in C^{1,\bar\alpha}(\overline\C_{1/2})$, with
$$
\|v\|_{C^{1,\bar\alpha}(\overline\C_{1/2})}\leq C \|v\|_{L^\infty(\C_1)},
$$
where $0<\bar\alpha<1$ and $C > 0$ depend only on $n$, $\lambda$, and $\Lambda$.
\end{Proposition}

\begin{proof}
By Theorem~\ref{thm:difference}, for any $h>0$ small and any unit vector $e'\in \R^n$ in the $x'$-direction, 
$$
 v(x+he', t) - v(x,t) \in \mathcal{S}(0) \quad \text{ and } \quad  v(x,t) - v(x,t-h) \in \mathcal{S}(0) \quad \text{ in } \C_{1-h}^\pm .
$$
By the arguments in \cite[Section~5.3]{Caffarelli-Cabre}, we get that $v|_{x_n=0} \in C^{1,\bar \alpha}(\overline\C_{3/4})$ for some $0< \bar \alpha<1$, depending only on $n$, $\lambda$, and $\Lambda$. Hence, by the boundary regularity estimates for the Dirichlet problem, it follows that
$$
\|v^\pm\|_{C^{1,\bar \alpha}(\overline\C_{1/2})}\leq C \|v\|_{L^\infty(\C_1)},
$$
where $C>0$ depends only on $n$, $\lambda$, and $\Lambda$.
In particular, $v$ satisfies the transmission condition in the classical sense, and thus, the estimate is propagated across $\C_1'$. 
\end{proof}

\section{Stability results}\label{Sec:Stab}

Let $|a|<1/2$, and define the hyperplane $T_a:=\C_1\cap \,\{x_n=a\}$.
 Call $\C_{1,a}^\pm :=\C_1\cap \{\pm \,(x_n-a)>0\}$, the domains above and below $T_a$.
Consider the flat interface problem
\begin{align}\label{eq:flat_Ta}
    \begin{cases}
        \partial_t v-F^\pm(D^2v)=0 & \mbox{ in } \C_{1,a}^\pm,\\
        v_{x_n}^+-v_{x_n}^-= g_0& \mbox{ on } T_a,
    \end{cases}
\end{align}
for some $g_0 \in \R$. Note that if $v|_{\partial_p \C_1}=:\varphi \in C(\partial_p \C_1)$, then by uniqueness (Corollary~\ref{Cor:Uniqueness}),
\begin{equation} \label{eq:repv}
    v= w+ \frac{g_0}{2} |x_n-a| \quad \text{ on } \overline\C_1,
\end{equation}
where $w \in C(\overline\C_1)$ is the viscosity solution to
\begin{equation} \label{pb:w}
\begin{cases}
\partial_t w - F^\pm(D^2 w) = 0 & \text{ in } \C_{1,a}^\pm,\\
w_{x_n}^+= w_{x_n}^- & \text { on } T_a,\\
w=\varphi - \frac{g_0}{2} |x_n-a| & \text{ on } \partial_p \C_1.
\end{cases}
\end{equation}

The $C^{1,\alpha}$ interior estimate for $w$ follows from Proposition~\ref{prop:regflat} and a standard covering argument.

\begin{Lemma} \label{lem:regflatw}
Let $|a|<1/2$, $g_0 \in \R$, and  $0<\eta<1$. Let $0<\bar \alpha<1$ be as in Proposition~\ref{prop:regflat}.
 Then the viscosity solution to \eqref{pb:w} satisfies $w\in C^{1,\bar\alpha}(\overline\C_{1-\eta})$, with
 $$
\eta \|\nabla w\|_{L^\infty(\C_{1-\eta})} + \eta^{1+\bar\alpha}[\nabla w]_{C^{0,\bar\alpha}(\overline\C_{1-\eta})} \leq C \big(\|\varphi\|_{L^\infty(\partial_p\C_1)}+ |g_0| \big),
$$
where $C>0$ depends only on $n$, $\lambda$, and $\Lambda$.
\end{Lemma}

In the next lemma, we call $w_0$ the viscosity solution to \eqref{pb:w} with $a=0$.

\begin{Lemma} \label{lem:stabilityw}
Let $|a|<1/2$ and $g_0 \in \R$. Assume that $\varphi \in C^\gamma(\partial_p \C_1)$, for some $0<\gamma<1$.
Let $w \in C(\overline\C_1)$ be the viscosity solution to \eqref{pb:w}.
For all $\vep>0$, there is $\delta>0$ such that if $|a|\leq \delta$,
 then
$$
\|w -w_0\|_{L^\infty(\C_1)} + \|\nabla w - \nabla w_0\|_{C^{0,\alpha}(\overline \C_{3/4})}\leq \vep,
$$
for any $0<\alpha<\bar\alpha$.
\end{Lemma}

\begin{proof}
Assume for the sake of contradiction that there are $\vep_0>0$ and functions $w_k$, for each $k\geq 1$, such that $w_k$ is the viscosity solution to \eqref{pb:w}, with $|a_k|\leq1/k$, 
but there is $0<\alpha<\bar\alpha$ such that
\begin{equation} \label{eq:contra}
\|w_k -w_0\|_{L^\infty(\C_1)} + \|\nabla w_k - \nabla w_0\|_{C^{0,\alpha}(\overline\C_{3/4})}> \vep_0.
\end{equation}
Note that for each $k\geq 1$, we have
$$
w_k = w_0 + \tfrac{g_0}{2} |x_n|  - \tfrac{g_0}{2} |x_n-a_k| \quad \text{ on } \partial_p \C_1.
$$
Let $\phi_k(x) := |x_n|  -|x_n-a_k|$. Then for any $x\in \overline\C_1$,
$$
|\phi_k(x)| \leq \big|  |x_n|  -|x_n-a_k| \big |\leq |a_k|\leq 1/k,
$$
and for any $x, y \in \overline\C_1$,
$$
|\phi_k(x)- \phi_k(y)| \leq   \big|  |x_n|  -  |y_n| \big| +\big| |x_n-a_k| - |y_n-a_k| \big | \big | \leq 2|x_n-y_n|. 
$$
It follows that
$$
\|w_k - w_0\|_{L^\infty(\partial_p \C_1)} = \tfrac{|g_0|}{2} \|\phi_k\|_{L^\infty(\partial_p \C_1)} \leq \tfrac{|g_0|}{2k} \to 0 \quad \text{ as } k \to \infty.
$$
By Proposition~\ref{Pro:Holder_global}, for $0<\beta \leq \min\{\alpha_1,\gamma/2\}$, we have
$$
\|w_k\|_{C^{0,\beta}(\overline\C_1)}  \leq C \|w_k\|_{C^{0,\gamma}(\partial_p \C_1)} \leq C ( \|\varphi\|_{C^{0,\gamma}(\partial_p \C_1)}+|g_0|),
$$ 
where $C>0$ depends only on $n$, $\lambda$, and $\Lambda$.
Moreover, by Lemma~\ref{lem:regflatw}, 
$$
\|w_k\|_{C^{1,\bar \alpha}(\overline\C_{3/4})} \leq C \|w_k\|_{L^\infty(\C_1)} \leq C (\|\varphi\|_{L^\infty(\partial_p \C_1)}+|g_0|).
$$
Hence, by Arzel\`{a}-Ascoli theorem, for any $0<\alpha<\bar \alpha$, up to a subsequence,
 \begin{align*}
 &w_k \to w_\infty \text{ uniformly on } \overline\C_1, \\
 & w_k \to w_\infty  \text{ in } C^{1,\alpha}(\overline\C_{3/4}).
 \end{align*}
By classical arguments, $w_\infty$ is a viscosity solution to \eqref{pb:w}, with $a=0$, such that $w_\infty = w_0$ on $\partial_p \C_1$. By uniqueness, it follows that $w_\infty=w_0$, which contradicts \eqref{eq:contra} for $k$ large.
\end{proof}

\begin{Lemma}[Stability for flat problems] \label{lem:stabilityflat}
    Let $0<\alpha<\bar \alpha$, $g_0\in \R$, and $\varphi \in C^\gamma(\partial_p \C_1)$, for some $0<\gamma<1$.
   Given $\vep>0$, let $0<\delta<1/2$ be as in Lemma~\ref{lem:stabilityw}. Let
     $\overline{v}, \underline v$ be viscosity solutions to \eqref{eq:flat_Ta} with $a=\delta$ and $a=-\delta$, respectively. Assume that $\underline v=\overline v= \varphi$ on $\p_p\C_1$. Then
     $$
\|\overline v -\underline v \|_{L^\infty(\C_1)} \leq 2\vep+|g_0|\delta.
$$
Furthermore, let $D_{\delta,\eta}:=\C_{1-\eta}\cap \{|x_n|<\delta\}$. For any $0<\eta<1$, it holds that
$$
\|\nabla' \overline v - \nabla' \underline v \|_{C^{0,\alpha}(\overline\C_{1-\eta})} +  \|(\underline v^+)_{x_n}-(\overline v^-)_{x_n}-g_0\|_{L^\infty(D_{\delta,\eta})} \leq \eta^{-(1+\alpha)} \vep.
 $$
\end{Lemma}

\begin{proof}
   Given $\vep>0$, let $0<\delta<1/2$ be as in Lemma~\ref{lem:stabilityw}. By \eqref{eq:repv}, recall that 
   \begin{equation} \label{eq:v}
\overline v = \overline w + \tfrac{g_0}{2}|x_n-\delta|  \quad \text{ and } \quad
\underline v = \underline w + \tfrac{g_0}{2}|x_n+\delta|,
\end{equation}
where $\overline w, \underline w \in C(\overline\C_1)$ are viscosity solutions to \eqref{pb:w} with $a=\delta$ and $a=-\delta$, respectively.
Let $w_0$ be the viscosity solution to \eqref{pb:w} with $a=0$.
By Lemma~\ref{lem:stabilityw}, it follows that
$$
\|\overline v -\underline v \|_{L^\infty(\C_1)} \leq 
\|\overline w -w_0 \|_{L^\infty(\C_1)} + \|\underline w -w_0 \|_{L^\infty(\C_1)}
+ |g_0|\delta \leq 2\vep + |g_0|\delta.
$$
Moreover, by the lemma (rescaled), for any $0<\eta<1$,
$$
 \|\nabla' \overline v - \nabla' \underline v \|_{C^{0,\alpha}({\overline\C_{1-\eta}})} =  \|\nabla' \overline w - \nabla' \underline w \|_{C^{0,\alpha}(\overline\C_{1-\eta})} \leq \eta^{-(1+\alpha)} \vep.
$$
Let $D_{\delta,\eta}:=\C_{1-\eta}\cap \{|x_n|<\delta\}$. On this domain, we have $\overline v=\overline v^-$ and $\underline v=\underline v^+$, where
$$
\overline v^- = \overline w + \tfrac{g_0}{2} (x_n-\delta) \quad\text{ and } \quad
\underline v^+ = \underline w + \tfrac{g_0}{2} (x_n+\delta). 
$$
Therefore, applying the lemma again, we get
\begin{align*}
  \|(\underline v^+)_{x_n}-(\overline v^-)_{x_n}-g_0\|_{L^\infty(D_{\delta,\eta})} 
 &=  \|\underline w_{x_n}-\overline w_{x_n}\|_{L^\infty(D_{\delta,\eta})} 
 \leq  \eta^{-(1+\alpha)} \vep.
\end{align*}
\end{proof}

Next, we establish the stability of the curved interface problem with respect to perturbations of the interface.
This is the key lemma to prove the $C^{1,\alpha}$ estimates in the next section. 

\begin{Lemma}[Stability for curved and flat problems]\label{Lm:Stab}
    Fix $0<\alpha<\bar \alpha$ and $0<\varepsilon<1$. Let $0<\delta<\vep/2$ be as in Lemma~\ref{lem:stabilityw}.
    Assume that $u\in C(\C_1)$ is a viscosity solution to \eqref{eq:main1} such that
    \[
    \|u\|_{L^\infty(\C_1)}\leq 1, \quad |g(0)|\leq 1, \quad \text{ and }\quad
    \|\psi\|_{\Coa(\overline{\C_1'})}+\|g-g(0)\|_{L^\infty(\Gamma)}+\|f\|_{L^{n+1}(\C_1)}\leq \delta.
    \]
        Set $g_0=g(0)$, and let $\overline v$ and $\underline v$ be given as in Lemma~\ref{lem:stabilityflat} with $\C_1$ replaced by $\C_{3/4}$. Define
    $$
    v :=\underline v \chi_{\Omega^+}+ \overline v \chi_{\overline \Omega^-} \quad \text{ on } \overline\C_1.
    $$ 
    If $v=u$ on $\p_p\C_{3/4}$, then there is $0<\tau<1/6$, depending only on $\alpha$ and $\alpha_1$, such that
    \[
    \|u-v\|_{L^\infty(\C_{1/2})}\leq C\varepsilon^{\tau},
    \]
    where $C>0$ depends only on  $n$, $\lambda$, $\Lambda$, and $\alpha$.
\end{Lemma}

\begin{proof}
   By Theorem \ref{thm:intholder} and the assumptions on $u, g, f^\pm$, we have that $u\in C^{0,\alpha_1}(\overline\C_{3/4})$ with
    \[
    \|u\|_{C^{0,\alpha_1}(\overline\C_{3/4})}\leq C_1.
    \]
    Since $v=u$ on $\p_p \C_{3/4}$, by Proposition~\ref{Pro:Holder_global}, it follows that $v^\pm\in C^{0,\beta}(\overline \Omega_{3/4}^\pm)$, with $\beta=\alpha_1/2$, and
    \[
    \|v\|_{C^{0,\beta}(\overline \Omega_{3/4}^\pm)}\leq C(1+\|u\|_{C^{0,\alpha_1}(\overline \C_{3/4})})\leq C_2.
    \]
    
    We emphasize that $v$ is discontinuous across $\Gamma_{3/4}$ since $\overline v-\underline v\not=0$ on $\Gamma_{3/4}$. Define $w\in C(\overline\C_{3/4})$ as the unique viscosity solution to the two following Cauchy-Dirichlet problems:
    \begin{align} \label{pb:CD}
        \begin{cases}
            \p_t w^\pm-F^\pm(D^2w^\pm)=0 & \mbox{ in } \Omega^\pm_{3/4}\\
            w=\frac{1}{2}(v^++v^-) &\mbox{ on } \Gamma_{3/4}\\
            w=v & \mbox{ on } \p_p \C_{3/4}.
        \end{cases}
    \end{align}
    Note that, by definition of $v$, we have
    $$
    v^+ = \underline v^+  \quad \text{ and } \quad v^- = \overline v_-.
    $$
 By Lemma~\ref{lem:regflatw} and \eqref{eq:v}, for any $0<\eta<3/4$, we have $v^\pm\in \Coa(\overline \Omega_{3/4-\eta}^\pm)$,  with 
    \begin{align}\label{eq:5.6}
        \eta\|\nabla v^\pm\|_{L^\infty(\Omega_{3/4-\eta}^\pm)}+\eta^{1+\alpha}[\nabla v^\pm]_{C^{0,\alpha}(\overline \Omega_{3/4-\eta}^\pm)}\leq C_3.
    \end{align}
Hence, by the pointwise $\Coa$ regularity for \eqref{pb:CD}, we get 
$$w^\pm\in \Coa(\overline \Omega_{3/4-\eta}^\pm).$$
Moreover, since $w-v \in \mathcal{S}(0)$ in $\Omega_{3/4}^\pm$, by the classical maximum principle, and Lemma~\ref{lem:stabilityflat},
    \begin{equation} \label{eq:w-v}
 \|w-v\|_{L^\infty(\C_{3/4})}\leq  \|w^\pm-v^\pm\|_{L^\infty(\Gamma_{3/4})}= \tfrac{1}{2}\|\bar v - \underline v\|_{L^\infty(\Gamma_{3/4})}\leq \vep+  \delta/2 \leq 2\vep.
    \end{equation}
    
Fix $0<\eta<1/4$ to be chosen later. On the other hand, we have that $u-w$ satisfies
    \begin{align*}
        \begin{cases}
            u-w\in \mathcal{S}(f^\pm) &\mbox{ in } \Omega_{3/4}^\pm\\
            (u-w)_\nu^+-(u-w)_\nu^-=\widetilde g & \mbox{ on } \Gamma_{3/4-\eta}\\
            u-w=0 &\mbox{ on } \p_p \C_{3/4},
        \end{cases}
    \end{align*}
    where the transmission condition holds in the viscosity sense, and
    $$
    \widetilde g := g - (w^+_\nu - w^-_\nu). 
    $$
We will see that $\widetilde g$ is small by showing that $w_\nu^+-w_\nu^-\approx g(0)$ on $\Gamma_{3/4-\eta}$. 
    Indeed, for any $x\in \Gamma_{3/4-\eta}$, 
    \begin{align*}
        w_\nu^+-w_\nu^--g(0)\,&=(w-v)_\nu^+-(w-v)_\nu^-+(v_\nu^+-v_\nu^--g(0)) \equiv g_1+g_2+g_3.
        \end{align*} 
 We proceed to show that $g_1, g_2, g_3$ are small for $\vep$ sufficiently small. For $g_3$, it holds that
    \begin{align*}
        |g_3|\,&\leq |v_\nu^+-v_{x_n}^+|+|v_{x_n}^+-v_{x_n}^--g(0)|+|v_\nu^--v_{x_n}^-|.    \end{align*}
    By Lemma~\ref{lem:stabilityflat}, $|v_{x_n}^+-v_{x_n}^--g(0)|\leq \eta^{-(1+\alpha)}\varepsilon$. Moreover, by \eqref{eq:5.6} and $|\nabla'\psi| \leq \delta$, we get
    \begin{align*}
        |v_\nu^\pm-v_{x_n}^\pm|\leq |\nabla v^\pm|\,|\nu-e_n| \leq C_3 \eta^{-1} |\nabla'\psi|\leq C_3 \eta^{-1}\delta.
    \end{align*}
    Hence, $|g_3|\leq C \eta^{-(1+\alpha)}\varepsilon.$\medskip
    
     Next, we bound $g_1$.  By pointwise boundary $\Coa$ estimates for the Pucci class \cite[Theorem~3.3]{Ma-Moreira-Wang}, 
    \begin{equation} \label{eq:g1}
    |g_1|\leq C\eta^{-1} \big(\|w^+-v^+\|_{L^\infty(\Omega_{3/4}^+)}+\|\overline v-\underline v\|_{\Coa(\overline \Gamma_{3/4-\eta})}\big),
    \end{equation}
    where
    \[
    \|v\|_{\Coa(\overline \Gamma_{3/4-\eta})} :=\|v\|_{L^\infty(\Gamma_{3/4-\eta})}+\|\nabla'v+v_{x_n}\nabla'\psi\|_{C^{0,\alpha}(\overline \Gamma_{3/4-\eta})}.
    \]
    Using again Lemma~\ref{lem:stabilityflat} together with \eqref{eq:5.6}, we get
    \begin{align} \nonumber
        \|\overline v-\underline v\|_{\Coa(\overline \Gamma_{3/4-\eta})}\leq&\, \|\overline v-\underline v\|_{L^\infty(\C_{3/4})}+\|\nabla'\overline v-\nabla'\underline v\|_{\Cza(\overline\C_{3/4-\eta})}+\|(\overline v_{x_n}-\underline v_{x_n})\nabla'\psi\|_{\Cza(\overline \Gamma_{3/4-\eta})}\\  \label{eq:diff}
        \leq &\,2\varepsilon +\delta +\eta^{-(1+\alpha)}\varepsilon+2C_3\eta^{-(1+\alpha)}\delta.
    \end{align}
   Combining \eqref{eq:w-v}, \eqref{eq:g1}, and \eqref{eq:diff}, we obtain $|g_1|\leq C\eta^{-(2+\alpha)}\varepsilon$. Similarly for $g_2$.\medskip
    
    Therefore, for any $(x,t)\in \Gamma_{3/4}$, it follows that
    \begin{align*}
|\widetilde g(x,t)| & \leq |g-g(0)| + |w_\nu^+-w_\nu^--g(0)|\\
& \leq \delta + |g_1| + |g_2|+ |g_3|\\
& \leq  C\eta^{-(2+\alpha)}\varepsilon.
    \end{align*}
    Furthermore, for any $(x,t)\in \partial_p \C_{3/4-\eta}$, since $u=w$ on $\partial_p \C_{3/4}$, we have
    \begin{align*}
    |u(x,t)-w(x,t)| \leq [u]_{C^{0,\alpha_1}(\overline\C_{3/4})} \eta^{\alpha_1} + [w]_{C^{0,\beta}(\overline\C_{3/4})} \eta^\beta \leq (C_1+C_2)\eta^\beta,
    \end{align*}
    where $\beta=\alpha_1/2$, and $C_1, C_2>0$ are given above.\medskip
    
By the maximum principle (Theorem~\ref{Thm:ABP}) applied to $u-w$ in $\Omega_{3/4-\eta}^\pm$, we get
\begin{align} \nonumber 
\|u-w\|_{L^\infty(\Omega_{3/4-\eta}^\pm)} 
& \leq \|u-w\|_{L^\infty(\partial_p \C_{3/4-\eta})} + C \big(\|\widetilde g\|_{L^\infty(\Gamma_{3/4-\eta})} + \|f\|_{L^{n+1}(\C_{3/4})} \big) \\ \label{eq:u-w}
& \leq (C_1+C_2) \eta^{\beta}+ C\eta^{-(2+\alpha)}\vep+C\vep.
\end{align}

Choose $\eta<\min\{1/4,\varepsilon^\gamma\}$ with $\gamma=\tfrac{1}{2(2+\alpha)}$. Combining \eqref{eq:w-v} and \eqref{eq:u-w}, we conclude that
$$
\|u-v\|_{L^\infty(\C_{1/2})} \leq \|u-w\|_{L^\infty(\C_{1/2})} +\|w-v\|_{L^\infty(\C_{1/2})}
\leq C ( \vep^{\gamma\beta} + \vep^{1/2} +\vep) 
\leq C \vep^{\gamma\beta}.
$$
\end{proof}

\section{$\Coa$ regularity: proof of the main theorem}\label{Sec:Coa}

In this section, we prove our main result (Theorem~\ref{thm:globalest}). The global $C^{1,\alpha}$ estimate is obtained by patching the classical interior estimates and the boundary estimates as usual. Hence, we only need to prove the pointwise $\Coa$ estimates at the interface (Theorem~\ref{Thm:main_boundary_Coa_regularity}).

The proof of this theorem will follow from iterating the next lemma.
\begin{Lemma}\label{Lm:geom_it}
    Given $0<\alpha <\bar \alpha$, there exist $0<\delta,\rho <1/2$, depending only on  $n$, $\lambda$, $\Lambda$, $\alpha$, $\bar \alpha$, such that for any viscosity solution $u$ to \eqref{eq:main1} satisfying
 \[
    \|u\|_{L^\infty(\C_1)}\leq 1, \quad |g(0)|\leq 1, \quad \text{ and }\quad
    \|\psi\|_{\Coa(\overline{\C_1'})}+\|g-g(0)\|_{L^\infty(\Gamma)}+C_{f^-}+C_{f^+}\leq \delta,
    \]
    there exist  affine functions $l^{\pm}(x) = A^{\pm} \cdot x + b$, with $|A^+| + |A^-| + |b| \leq C_0$, such that
\[ \|u^{\pm} - l^{\pm}\|_{L^\infty(\Omega^\pm_\rho)} \leq \rho^{1+\alpha}, \]
where $C_0>0$ depends only on $n$, $\lambda$, and $\Lambda$.
Moreover, $A^+ - A^- = g(0)e_n$. 
\end{Lemma}
\begin{proof}
    Fix $0<\varepsilon<1$ and $0<\rho<1/2$ to be chosen. Let $0<\delta<\vep/2$ be as in Lemma~\ref{lem:stabilityw}, and let $v$ be as in Lemma~\ref{Lm:Stab}. By construction, $v^+=\underline v$ and $v^-=\overline v$, where $ \underline v$ and $\overline v$ satisfy \eqref{eq:flat_Ta}, replacing $\C_1$ by $\C_{3/4}$, with $a = -\delta$ and $a = \delta$, respectively.  Consider
\[ 
l_v^\pm(x) := \nabla v^\pm(0) \cdot x +   v^\pm (0).
\]
By Lemma~\ref{lem:regflatw} and \eqref{eq:repv}, we have $\overline v \in C^{1,\bar\alpha}(\C_{1/2} \cap \{x_n \leq \delta\})$, and
\[ 
\|\overline v\|_{C^{1,\bar\alpha}(\C_{1/2} \cap \{x_n \leq \delta\})} \leq C_0, 
\]
where $C_0 > 0$ depends only on $n$, $\lambda$, and $\Lambda$. Similarly, $\underline v \in C^{1,\bar\alpha}(\C_{1/2} \cap \{x_n \geq -\delta\})$, and
\[
\|\underline v\|_{C^{1,\bar\alpha}(\C_{1/2} \cap \{x_n \geq -\delta\})} \leq C_0.
\]
Hence, we have $|\nabla v^\pm(0)|+|v^\pm(0)|\leq C_0$. 

 To satisfy the transmission condition, we add a small perturbation to $l_v^\pm$. Namely, we define the affine functions $l^\pm(x)=A^\pm \cdot x + b$ as
$$
l^\pm := l_v^\pm \mp \tfrac{1}{2} l_\vep,
$$
where $l_\vep(x)= A_\vep \cdot x + b_\vep$ is given by
$$
    b_\vep :=  v^+(0)-v^-(0) \quad \text{ and } \quad A_\vep := \nabla v^+(0)-\nabla v^-(0)-g(0)e_n.
$$
Note that $l^+(0)=l^-(0)$ and $A^+-A^- = \nabla v^+(0) - \nabla v^-(0) - A_\vep= g(0)e_n.$
Moreover, by Lemma~\ref{lem:stabilityflat}, using that $|g(0)|\leq 1$, and $\delta\leq \vep/2$, we have
\begin{align} \label{eq:error}
|b_\vep| \leq 3\vep \quad \text{ and } \quad |A_\vep| \leq 2^{1+\alpha} \vep.
\end{align}
Since $\rho< 1/2$, it follows that
\begin{align*}
    \|u^+-l^+\|_{L^\infty(\Omega_\rho^+)} 
    &\leq \|u^+-v^+\|_{L^\infty(\Omega_\rho^+)}+\|v^+-l_v^+\|_{L^\infty(\Omega_\rho^+)}+\tfrac{1}{2}|A_\vep| \rho + \tfrac{1}{2}|b_\vep|\\ 
    &\leq C\varepsilon^\tau+C_0\rho^{1+\bar \alpha}+ \big(\tfrac{3}{2}+2^\alpha\big)\varepsilon,
\end{align*}
where in the last inequality, the first estimate follows from Lemma~\ref{Lm:Stab}, the second one from the $C^{1,\bar\alpha}$ regularity of $v^+$, and the third one by \eqref{eq:error}. Choosing
\[
\rho < \min \left\{ \frac{1}{2}, \Big(\frac{1}{3C_0}\Big)^\frac{1}{\bar\alpha- \alpha}\right\} \quad \text{ and } \quad \varepsilon\leq \min\left\{\Big(\frac{\rho^{1+\alpha}}{3C}\Big)^\frac{1}{\tau},\frac{\rho^{1+\alpha}}{1/2+2^\alpha/3}\right\},
\]
we get $\|u^+-l^+\|_{L^\infty(\Omega_\rho^+)}\leq \rho^{1+\alpha}$, as intended.
The estimate in $\Omega_\rho^-$ follows in a similar way.
\end{proof}

For the next proof, we define the pointwise H\"{o}lder semi-norms
$$
[\psi]_{C^{1,\alpha}(0)} := \sup_{(x',t) \in \C_1'\setminus\{0\}} \frac{|\psi(x',t)-\psi(0)-\nabla '\psi(0)\cdot x'|}{(|x'|^2+ |t|)^{(1+\alpha)/2}},
$$
and
$$
[g]_{C^{0,\alpha}(0)} := \sup_{\substack{(x,t) \in \Gamma\setminus\{0\}}} \frac{|g(x,t)-g(0)|}{(|x|^2+|t|)^{\alpha/2}}.
$$

\begin{customproof}{Theorem~\ref{Thm:main_boundary_Coa_regularity}}
Let \(0<\alpha<\bar\alpha\). Let  $0< \delta, \rho<1/2$ be as in Lemma~\ref{Lm:geom_it}, and fix $\widetilde C>0$ to be chosen large.
Let 
\[
0<\delta_0<\min \left\{ \frac{\delta}{3}, \frac{\rho^{1+\alpha}}{2 \widetilde C}, \frac{\delta}{2+4\widetilde C} \right\}.
\]
First, we normalize the problem. Recall that we are assuming \(0\in\Gamma\), i.e., \(\psi(0)=0\). \medskip

\begin{enumerate}
\item[$(i)$] After a rotation, we can assume that \(\nu(0)=e_n\). In particular, \(\nabla'\psi(0)=0\). Also, we can suppose that \([\psi]_{C^{1,\alpha}(0)} \leq \delta_0\).
Indeed, let \(K^\alpha=[\psi]_{C^{1,\alpha}(0)}/\delta_0\), and consider the function $v(y,s)=u(y/K,s/K^2)$ for $(y,s)\in \C_1.$ Then $v$ satisfies
\[
\begin{cases}
\p_s v-F_K^\pm(D^2v) =  f_K^\pm & \text{in } \widetilde \Omega^\pm,\\
v_\nu^+- v_\nu^- = g_K & \text{on } \widetilde \Gamma,
\end{cases}
\]
where  \(\widetilde \Omega^\pm = \{ (y,s) \in \C_1 : (y/K,s/K^2) \in \Omega^\pm\}\),  \(\widetilde \Gamma = \{ (y,s) \in \C_1 :(y/K,s/K^2) \in \Gamma\}\), \(F_K^\pm(M) = K^{-2} F^\pm(K^2 M)\) for  \(M\in \mathcal{S}(n)\),  \(f_K^\pm(y,s)=K^{-2} f^\pm(y/K,s/K^2)\) for \((y,s)\in \widetilde \Omega^\pm \),  and \(g_K(y,s)=K^{-1} g(y/K,s/K^2)\) for \((y,s)\in \widetilde \Gamma\).  
Then \(F_K^\pm\) are uniformly \( (\lambda,\Lambda) \)-elliptic. Moreover,
\[
\left( \fint_{\widetilde\Omega^\pm_r} |f_K^\pm(y,s)|^{n+1}\, dyds \right)^\frac{1}{n+1} 
\leq C_{f^\pm_K} r^{\alpha-1}, \quad \text{ for all $r>0$ small},
\]
with \(C_{f^\pm_K} = K^{-(1+\alpha)} C_{f^\pm}\),  and
\[
[g_K]_{C^{0,\alpha}(0)} := \sup_{\substack{(y,s) \in \widetilde \Gamma\setminus\{0\}}} \frac{|g_K(y,s)-g_K(0)|}{(|y|^2+|s|)^{\alpha/2}}
\leq K^{-(1+\alpha)} [g]_{C^{0,\alpha}(0)}.
\]
If \((y,s)\in \widetilde\Gamma\), then \(y_n=\widetilde \psi(y',s)\), with \(\widetilde{\psi}(y',s) = K \psi (y'/K,s/K^2)\), and 
\[
[\widetilde \psi]_{C^{1,\alpha}(0)} 
\leq K^{-\alpha} [\psi]_{C^{1,\alpha}(0)}= \delta_0.
\]

\item[$(ii)$] We can assume that \(\| u \|_{L^\infty(\C_1)}\leq 1\), $|g(0)|\leq 1$, and 
\[
  [g]_{C^{0,\alpha}(0)} + C_{f^-}+C_{f^+} \leq \delta_0.
\]
Indeed, let \(K=\| u \|_{L^\infty(\C_1)}+ |g(0)|+ \delta_0^{-1}([g]_{C^{0,\alpha}(0)} +C_{f^-}+C_{f^+})\), and consider \(v =u/K\). Then \(v\) satisfies
\[
\begin{cases}
\p_tv-F_K^\pm(D^2v) =  f_K^\pm & \text{in } \Omega^\pm,\\
v_\nu^+- v_\nu^- = g_K & \text{on } \Gamma,
\end{cases}
\]
where \(F_K^\pm(M) = K^{-1} F^\pm(K M)\), \(f_K^\pm=K^{-1} f^\pm\), and \(g_K=K^{-1} g\).  
Moreover, \(\|v\|_{L^\infty(\C_1)}\leq1\), $|g_K(0)|\leq 1$, and \([g_K]_{C^{0,\alpha}(0)} + C_{f_K^-}+C_{f_K^+}\leq \delta_0\).
\end{enumerate}
For simplicity, we keep the same notation, i.e., \(\psi\), \(u\), \(F^{\pm}\), \(f^\pm\), and \(g\).\medskip

Under these assumptions, it suffices to prove the following.

\medskip
\noindent\textbf{Claim.} For every \(k \geq 1\), there exist affine functions \(l_k^\pm(x)=A_k^\pm \cdot x + b_k\) such that
\[
\rho^{k-1} |A^\pm_{k}-A^\pm_{k-1}|  + |b_{k}-b_{k-1}|  \leq C_0 \rho^{(k-1)(1+\alpha)},
\]
\[
 A_k^+-A_k^-=g(0)\, e_n,
\]
where \(C_0>0\) depends only on \(n\), \(\lambda\), \(\Lambda\), and \(\alpha\), and such that
\[
\|u^\pm-l_k^\pm \|_{L^\infty( \Omega^\pm_{\rho^k})} \leq \rho^{k(1+\alpha)}.
\]

We prove the claim by induction. For $k=1$, by the normalization, we are under the assumptions of Lemma~\ref{Lm:geom_it}. Indeed, by $(i)$, we have that 
\begin{align*}
\|\psi\|_{C^{1,\alpha}(\overline{\C_1'})} &= \|\psi-\psi(0)\|_{L^\infty(\C_1')} + \|\nabla' \psi-\nabla'\psi(0)\|_{L^\infty(\C_1')} + [\nabla\psi]_{C^{0,\alpha}(\overline{\C_1'})}\\
&\leq 3[\psi]_{C^{1,\alpha}(0)} \leq 3 \delta_0 \leq \delta.
\end{align*}
Moreover, by $(ii)$, we have \(\| u \|_{L^\infty(\C_1)}\leq 1\), $|g(0)|\leq 1$, and
\[  \|g-g(0)\|_{L^\infty(\Gamma)} + C_{f^-}+C_{f^+} \leq [g]_{C^{0,\alpha}(0)} +C_{f^-}+C_{f^+} \leq \delta_0 \leq \delta/3. \]
Hence, by Lemma~\ref{Lm:geom_it}, there exist $l_1^\pm (x) = A_1^\pm \cdot x + b_1$, with 
$$
|A_1^\pm - A_0^\pm| +|b_1-b_0| \leq C_0, 
$$
$$
A_1^+-A_1^-= g(0)e_n,
$$
where $A_0^\pm = 0, \ b_0=0$, such that
\begin{align*}
\|u^\pm -l_1^\pm \|_{L^\infty(\Omega_\rho^\pm)} &\leq \rho^{1+\alpha}.
\end{align*}

For the induction step, assume that the claim holds for some $k\geq 1$, and let $l^\pm_k$ be such affine functions. Denote by
\[ \widetilde {\Omega}^\pm_{k} =\{(x,t)\in \C_1 : (\rho^k x,\rho^{2k} t) \in \Omega^\pm\}\quad\hbox{and}\quad \widetilde \Gamma_{k} =\{ (x,t) \in \C_1 : (\rho^k x,\rho^{2k} t) \in \Gamma \}. \]
Note that 
if $\widetilde \Gamma_k$ is the graph of $\psi_k$ in $\C_1'$, then $\psi_{k}(x',t) = \rho^{-k}\psi(\rho^k x',\rho^{2k} t).$
In particular, $\nabla' \psi_{k} (x',t) =\nabla' \psi(\rho^k x,\rho^{2k} t) $, and thus, for $(x,t)\in \widetilde\Gamma_{k}$, if $\nu_{k}(x,t)$ is the normal pointing at $\widetilde\Omega_{k}^+$, then $\nu_{k}(x,t) = \nu(\rho^k x,\rho^{2k} t) $.
Define $l_k= l_k^+ \chi_{\widetilde\Omega_{k}^+} + l_k^- \chi_{\widetilde\Omega_{k}^-}$.
Consider the rescaled function
\[ v(x,t) = \frac{u(\rho^k x,\rho^{2k} t) - {l}_k(\rho^k x,\rho^{2k} t)}{\rho^{k(1+\alpha)}}\qquad\hbox{for}~(x,t)\in \overline\C_1. \]
Then $v$ satisfies
\begin{equation*} 
\begin{cases}
\p_tv-F_k^\pm(D^2 v) = f_k^\pm & \hbox{in}~\widetilde\Omega_{k}^\pm,\\
v_{\nu_k}^+ - v_{\nu_k}^- = g_k & \hbox{on}~\widetilde\Gamma_{k}\\
\end{cases}
\end{equation*}
in the viscosity sense, where 
\begin{align*}
F_k^\pm(M) &= \rho^{k(1-\alpha)} F^\pm( \rho^{k(\alpha-1)}M) \quad\hbox{for}~M\in \mathcal{S}^n,\\
f_k^\pm(x,t) &= \rho^{k(1-\alpha)} f^\pm(\rho^k x,\rho^{2k} t) \quad \hbox{for}~(x,t)\in \widetilde\Omega_{k}^\pm,\\
g_k(x,t) & = \rho^{- k \alpha} (g(\rho^k x,\rho^{2k} t)-g(0)\nu_n(\rho^k x,\rho^{2k} t)) \quad \hbox{for}~(x,t)\in \widetilde\Gamma_{k}.
\end{align*}
By the induction hypothesis,
$\| v \|_{L^\infty(\C_1)} \leq 1.$
We also have that
\begin{equation} \label{eq:cttf}
\Big( \fint_{\C_r \cap \widetilde\Omega_k^\pm} |f_k^\pm(x,t)|^{n+1}\, dxdt\Big)^\frac{1}{n+1}
\leq C_{f^\pm} r^{\alpha-1}.
\end{equation}
Hence, $C_{f_k^\pm} = C_{f^\pm}$, and $C_{f_k^-}+C_{f_k^+} \leq \delta_0$.
Moreover,
\begin{equation} \label{eq:estg}
\|g_k\|_{L^\infty(\widetilde\Gamma_k)} \leq [g]_{C^{0,\alpha}(0)}+ [\nu_n]_{C^{0,\alpha}(0)} \leq \delta_0+\delta_0=2\delta_0.
\end{equation}
However, we cannot apply Lemma~\ref{Lm:geom_it} to $v$ since it has a jump discontinuity on $\widetilde\Gamma_{k}$. In fact, if
$v^\pm = v \big|_{\overline{\widetilde \Omega^\pm_{k}}},$
then for $(x,t)\in \widetilde\Gamma_{k}$, by the normalization $(i)$, and the induction hypothesis, we have
\begin{align} \label{eq:jump}
|(v^- - v^+)(x,t)| &= \frac{|l_k^-(\rho^k x)-l_k^+(\rho^k x)|}{\rho^{k(1+\alpha)}} = \rho^{-k \alpha} |g(0)| |x_n| \nonumber \\
&\leq \rho^{-k \alpha} \sup_{(x,t) \in \widetilde\Gamma_{k}} |x_n|  \\
&\leq \sup_{(x',t)\in \C_1'} \frac{|\psi_{k}(x',t)|}{ \rho^{k \alpha} } \leq [\psi]_{C^{1,\alpha}(0)} \leq \delta_0.\nonumber
\end{align}

Let $w\in C(\C_1)$ be the viscosity solution of the Dirichlet problems:
\begin{equation*}
\begin{cases}
\p_tw-F_k^\pm (D^2 w) = 0 & \hbox{in}~\widetilde\Omega_{k}^\pm,\\
w = \frac{1}{2}(v^++v^-) & \hbox{on}~\widetilde\Gamma_{k},\\
w=v & \hbox{on}~\p_p \C_{1}.
\end{cases}
\end{equation*}
We will prove that $w$ satisfies the assumptions of Lemma~\ref{Lm:geom_it}. Indeed, by the maximum principle, $\|w\|_{L^\infty(\C_1)} \leq \| v\|_{L^\infty(\p_p\C_1)}\leq 1$.
Moreover, $v^\pm - w^\pm \in \mathcal{S}(f^\pm_k)$ in $\widetilde \Omega_k^\pm$, $v^\pm - w^\pm = \pm \tfrac{1}{2} \rho^{-k\alpha} |g(0)| x_n$ on $\widetilde \Gamma_k$, and $v^\pm - w^\pm =0$ on $\partial_p \widetilde \Omega_k^\pm \setminus \widetilde \Gamma_k.$
Then, by the classical ABP estimate, \eqref{eq:jump}, and $(ii)$,
\begin{equation} \label{eq:smalldiff}
\|v^\pm - w^\pm \|_{L^\infty(\widetilde\Omega_k^\pm)} \leq \|v^\pm - w^\pm \|_{L^\infty(\widetilde\Gamma_k)} + C \|f_k^\pm\|_{L^{n+1}(\widetilde\Omega_k^\pm)} \leq \widetilde C \delta_0,
\end{equation}
taking $\widetilde C>0$ large. 
By boundary pointwise $C^{1,\alpha}$ estimates \cite{Ma-Moreira-Wang}, for any $(x_0,t_0)\in \widetilde\Gamma_k\cap \C_{3/4}$, 
\begin{align} \label{eq:regdiff}
|\nabla(v^\pm - w^\pm)(x_0,t_0)| & \leq {C} \big (\|v^\pm - w^\pm \|_{L^\infty(\widetilde\Omega_k^\pm)} + \tfrac{1}{2} \rho^{-k\alpha} \|\psi_k \|_{C^{1,\alpha}(x_0,t_0)} +C_{f_k^\pm} \big) \leq \widetilde C\delta_0,
\end{align}
where the last inequality follows from \eqref{eq:cttf}, \eqref{eq:smalldiff}, and the normalization $(i)$. 

Let $(x_0,t_0)\in \widetilde \Gamma_{k} \cap \C_{3/4}$. Assume that $\varphi$ touches $w$ by above at $(x_0,t_0)$ in a small neighborhood of $(x_0,t_0)$ contained in $\C_{3/4}$.
Then $\phi = \varphi - (w-v)$ is touches $v$ by above at $(x_0,t_0)$, and
\[ \phi_{\nu_k}^+(x_0,t_0) - \phi_{\nu_k}^-(x_0,t_0) \geq g_k(x_0,t_0). \]
Therefore, at $(x_0,t_0)$, we have that 
\[ \varphi_{\nu_k}^+ - \varphi_{\nu_k}^- \geq g_k + (w^+-v^+)_{\nu_k}-(w^--v^-)_{\nu_k} =: \widetilde{g}_k. \]
Moreover, by \eqref{eq:regdiff}, we get $|\widetilde g_k(0)|\leq 2\widetilde C\delta_0\leq \rho^{1+\alpha}<1$ (note $g_k(0)=0)$, and together with \eqref{eq:estg}, 
\begin{align} \label{eq:gk}
\|\widetilde g_k - \widetilde g_k(0)\|_{L^\infty(\widetilde \Gamma_{k} \cap \C_{3/4})}
& \leq 2\delta_0 + 4\widetilde C \delta_0 \leq \delta.
\end{align}
Similarly, if $\varphi$ is a test function touching $w$ from below at $(x_0,t_0)$, then
\[ \varphi_{\nu_k}^+(x_0,t_0) - \varphi_{\nu_k}^-(x_0,t_0) \leq \widetilde g_k(x_0,t_0). \]
Hence, $w_{\nu_k}^+ - w_{\nu_k}^- = \widetilde g_k$ on $\widetilde\Gamma_{k} \cap \C_{3/4}$ in the viscosity sense.
Applying Lemma~\ref{Lm:geom_it} to $w$, there is a constant $\widetilde C_0>0$, depending only on $n$, $\lambda$, $\Lambda$, and $\alpha$, and
affine functions $\widetilde l^\pm(x)= A^\pm\cdot x + \widetilde b$, with $|\widetilde A^\pm|+|\widetilde b|\leq  \widetilde C_0$, and $\widetilde A^+-\widetilde A^-=\widetilde g_k(0) e_n$, such that
\begin{align} \label{eq:estw_1}
\| w - \widetilde l^\pm \|_{L^\infty(\Omega_\rho^\pm)} &\leq \rho^{1+\alpha}/4.
\end{align}
Note that the estimate holds with $\rho^{1+\alpha}/4$ instead of $\rho^{1+\alpha}$, taking $\rho$ smaller in the proof of Lemma~\ref{Lm:geom_it}. 
Since $\widetilde g_k(0)\neq 0$, in general, we need to correct the polynomials. Define
$$
\hat l^\pm := \widetilde l^\pm \mp \tfrac{\widetilde g_k(0)}{2} x_n = \hat A^\pm \cdot x + \widetilde b.
$$
Then $\hat A^+_n - \hat A^-_n = (\widetilde A_n^+-\widetilde A_n^-)-\big( \tfrac{\widetilde g_k(0)}{2}+\tfrac{\widetilde g_k(0)}{2}\big)=\widetilde g_k(0)-\widetilde g_k(0)=  0$, and by \eqref{eq:estw_1}, 
$$
\| w - \hat l^\pm \|_{L^\infty(\widetilde \Omega_\rho^\pm)}  \leq \| w - \widetilde l^\pm \|_{L^\infty(\widetilde \Omega_\rho^\pm)} + \tfrac{|\widetilde g_k(0)|}{2} \rho\leq \rho^{1+\alpha}/4 +  \rho^{1+\alpha}/4 \leq \rho^{1+\alpha}/2.
$$
Hence, by \eqref{eq:smalldiff} and \eqref{eq:estw_1}, we get 
\begin{align*}
\| v - \hat l^\pm \|_{L^\infty(\widetilde \Omega^\pm_\rho)} & \leq \| v- w \|_{L^\infty(\C_\rho)} + \| w - \hat l^\pm \|_{L^\infty(\widetilde \Omega^\pm_\rho)} \leq  \widetilde C \delta_0 +\rho^{1+\alpha}/2 \leq \rho^{1+\alpha}.
\end{align*}
In particular, for any $(x,t)\in \widetilde \Omega^\pm_\rho$, we have
\begin{align*}
\left|\frac{u^\pm(\rho^k x,\rho^{2k}t) - {l}_k^\pm(\rho^k x)}{\rho^{k(1+\alpha)}} - \hat l^\pm(x) \right| \leq \rho^{1+\alpha},
\end{align*}
or equivalently, if $(y,s) = (\rho^k x,\rho^{2k}t) $, then for any $(y,s) \in \Omega^\pm_{\rho^{k+1}}$,
\begin{equation} \label{eq:indstep}
| u^\pm(y,s) - l_k^\pm(y) - \rho^{k(1+\alpha)} \hat l^\pm(\rho^{-k}y)| \leq \rho^{(k+1)(1+\alpha)}.
\end{equation}

Define the affine approximations at the step $k+1$ as
\[ l_{k+1}^\pm (y) = l_k^\pm (y) + \rho^{k(1+\alpha)} \hat l^\pm(\rho^{-k}y). \]

If $l_{k+1}^\pm (y) = A_{k+1}^\pm \cdot y+ b_{k+1}$, then
$A_{k+1}^\pm = A_k^\pm + \rho^{k\alpha}\hat A^\pm$ and $b_{k+1} = b_k + \rho^{k(1+\alpha)} \widetilde b$.
By the induction hypothesis, since $|\widetilde A|+|\widetilde b|\leq \widetilde C_0$, $\hat A^\pm = \widetilde A^\pm \mp\frac{\widetilde g_k(0)}{2}e_n$, and $\hat A^+-\hat A^-=0$, we get
\begin{align*}
\rho^{k} |A^\pm_{k+1}-A^\pm_{k}| + |b_{k+1}-b_{k}| \leq C_0 \rho^{k(1+\alpha)},
\end{align*}
$$
A_{k+1}^+-A_{k+1}^- = g(0)e_n,
$$
taking $C_0>0$ large enough, depending only on $n$, $\lambda$, $\Lambda$, and $\alpha$.
Moreover, by \eqref{eq:indstep}, we see that
\[ \| u^\pm - l_{k+1}^\pm \|_{L^\infty(\Omega^\pm_{\rho^{k+1}})} \leq \rho^{(k+1)(1+\alpha)}. \]
By induction, the claim holds for all $k\geq 1$, and we conclude the proof.
\end{customproof}

\appendix
\section{Hopf lemma for fully nonlinear parabolic equations} \label{app:hopf}

In this section, we prove a Hopf lemma for fully nonlinear parabolic equations. In what follows, we
call a constant universal if it depends only on $n$, $\lambda$, and $\Lambda$. 

\begin{Definition}[Dini function]
    A function $\omega : [0,\infty)\to [0,\infty)$ is called a Dini function if $\omega$ is nondecreasing and there is some $r_1>0$ such that
    $$
    \int_0^{r_1} \frac{\omega(r)}{r}\, dr< \infty.
    $$
\end{Definition}

\begin{Definition}[Interior $C^{1,\dini}$ condition] \label{def:intdini}
We say that $\Omega\subset \R^{n+1}$ satisfies the interior $C^{1,\dini}$ condition at $0\in \p_p \Omega$ if there exist $\bar r>0$ and a Dini function $\omega_\Omega$ such that for all $0<r\leq\bar r$,
$$
\C_{r} \cap \big\{x_n >  r \omega_\Omega(r) \big\} \subset \C_{r}\cap \Omega.
$$
\end{Definition}

\begin{Theorem}[Hopf lemma] \label{thm:hopf}
   Let $\Omega \subset\R^{n+1}$ be a domain satisfying the interior $C^{1,\dini}$ condition at $0\in \p_p \Omega$. Assume that $\bar r\geq 1$ and $\omega_\Omega(1)\leq 1/4$. Let $u$ satisfy
    $$
    \begin{cases}
        u \in {S}(0) \text{ in } \Omega\cap\C_1,\\
        u\geq 0 \text{ in } \Omega\cap\C_1,\\
        u(0)=0.
    \end{cases}
    $$
    Then for any $l\in \p B_1$ and $l_n=l\cdot e_n>0$, we have that
    \begin{equation} \label{eq:hopfest}
    u(rl,0)\geq cl_n u(e_n/2,-3/4)r,
    \end{equation}
    for all $0<r<r_0$, where $c>0$ and $r_0$ depend only on $n$, $\lambda$, $\Lambda$, and $\omega_\Omega$, and $r_0$ depends on $l$.
\end{Theorem}

We will use the following lemma proved in \cite[Lemma~2.1]{Lian-Zhang-2022}.

\begin{Lemma} \label{lem:flathopf}
    Let $u$ satisfy
    $$
\begin{cases}
    u\in S(0) & \text{ in } \C_1^+,\\
        u \geq 0 & \text{ in } \C_1^+,\\
    u = 0 & \text{ on } \C_1'.
\end{cases}
    $$
Assume that $u(e_n/2,-3/4)=1$. Then there is a universal constant $C>0$ such that
$$
u(x,t) \geq C x_n \quad \text{ in } \C_{1/2}^+.
$$
\end{Lemma}

We also need the $C^{1,\alpha}$ boundary estimate for solutions to problems with flat boundaries given in \cite[Theorem~2.8]{Lian-Zhang-2022}.

\begin{Lemma} \label{lem:flatbdryreg}
    Let $u$ satisfy
   $$
    \begin{cases}
        u \in S(0) & \text{ in } \C_1^+,\\
        u=0 & \text{ on } \C_1'.
    \end{cases}
   $$
   Then $u\in C^{1,\bar \alpha}(0)$, i.e., there exists a constant $\bar a$ such that
   $$
    |u(x,t) - \bar a x_n| \leq C_1  x_n (|x|^2+|t|)^{\bar \alpha/2} \|u\|_{L^\infty(\C_1^+)}, \quad \text{for all } (x,t)\in \C_{1/2}^+,
   $$
   and 
   $$
|\bar a|\leq C_1 \|u\|_{L^\infty(\C_1^+)},
$$
where $0<\bar\alpha <1$ and $C_1>0$ are universal constants.
\end{Lemma}

The proof of the Hopf Lemma follows the same lines as its elliptic counterpart \cite[Theorem~1.17]{Lian-Zhang-2023}.\medskip

\begin{customproof}{Theorem~\ref{thm:hopf}}
     Let $\Omega \subset\R^{n+1}$ be a domain satisfying the interior $C^{1,\dini}$ condition at $0\in \p_p \Omega$. Assume that $\bar r\geq 1$ and $\omega(1)\leq 1/4$.
     There exists $r_1>0$ such that
     \begin{equation} \label{eq:dinicond}
    \omega(r_1)\leq c_0 \quad \text{and} \quad \int_0^{r_1} \frac{\omega(r)}{r}\, dr \leq c_0,
     \end{equation}
     where we denoted $\omega=\omega_\Omega$, and $c_0\leq 1/4$ is some small universal constant to be determined. After rescaling, we may assume that $r_1=1$ and $u(e_n/2, -3/4) =1$. For $r>0$, let $\Omega_r^+ = \Omega \cap \C_r^+$. To prove \eqref{eq:hopfest}, it is enough to see that there are universal constants $0<\alpha_0, \rho <1$, $\widetilde C, \widetilde a>0$, and a nonnegative sequence $\{a_k\}$ such that for all $k\geq 0$,
     \begin{align}\label{eq:lowerest}
         \inf_{\Omega^+_{\rho^{k+1}}} (u - \widetilde a x_n + a_k x_n ) \geq - \rho^k A_k,\\ \label{eq:cauchy}
               a_k - a_{k-1} \leq \widetilde C A_k,\\ \label{eq:boundak}
               a_k \leq \frac{\widetilde a}{2},
     \end{align}
     where $a_{-1}=0$,  $A_0 = c_0$, and $A_{k} = \max\{\omega(\rho^k), \rho^{\alpha_0}A_{k-1} \}$ for all $k\ge 1$.

     Indeed, from the definition of $A_k$, we have
     $$
     \sum_{k=0}^\infty A_k \leq c_0 + \sum_{k=1}^\infty \omega(\rho^k) + \rho^{\alpha_0} \sum_{k=0}^\infty A_k,
     $$
     and thus, together with the Dini condition in \eqref{eq:dinicond}, we obtain 
\begin{align*}
\sum_{k=0}^{\infty} A_k 
&\leq \frac{1}{1 - \rho^{\alpha_0}} \Big(c_0+\sum_{k=1}^{\infty} \omega(\rho^k)\Big)\\
& = \frac{1}{1 - \rho^{\alpha_0}}\Big( c_0 + \frac{1}{1 - \rho} \sum_{k=1}^{\infty} \omega(\rho^k) \frac{\rho^{k-1} - \rho^k}{\rho^{k-1}} \Big)\\
&\leq \frac{1}{1 - \rho^{\alpha_0}} \Big( c_0 + \frac{1}{1 - \rho} \int_0^1 \frac{\omega(r)}{r}\, dr \Big)\\
&\leq \frac{c_0}{1 - \rho^{\alpha_0}}\Big(1+ \frac{1}{1 - \rho}\Big)\\
& \leq \frac{2c_0}{(1 - \rho^{\alpha_0})(1-\rho)} \leq 4c_0,
\end{align*}
provided $(1 - \rho^{\alpha_0}) (1 - \rho) \geq 1/2.$       
In particular, $A_k\to 0$ as $k\to \infty$. 
Let $l\in \p B_1$ with $l_n>0$. Then, there exists $k_0\geq 0$ such that
     \begin{equation} \label{eq:estAk}
         \frac{A_k}{\rho^2} \leq \frac{\widetilde a l_n}{4} \qquad \text{for all } k\geq k_0.
     \end{equation}
     Take $r_0 = \rho^{k_0}$. For $0<r<r_0$, there exists $k\geq k_0$ such that $\rho^{k+2}\leq r < \rho^{k+1}$. Then by \eqref{eq:lowerest}, \eqref{eq:boundak}, and \eqref{eq:estAk}, we have that
     $$
    u(rl, 0) \geq \widetilde a l_n r - a_k l_n r -\rho^k A_k \geq \frac{\widetilde a l_n r}{2} - \frac{A_k r}{\rho^2} \geq \frac{\widetilde a l_n}{4}r.
     $$
     which is the desired estimate in \eqref{eq:hopfest}. 
     
     We will prove \eqref{eq:lowerest}--\eqref{eq:boundak} by induction. Let $\widetilde Q=\C_{1/2}^+(c_0e_n,0)$ and $\widetilde T = \p_p \widetilde Q \cap \{x_n=1/2+c_0\}$. Since $c_0\leq 1/4$, then $\widetilde Q \subset \Omega\cap\C_1^+$. By the parabolic Harnack inequality,
     $$
    \inf_{\widetilde T} u \geq \widetilde c  u(e_n/2, -3/4) = \widetilde c,
     $$
     where $\widetilde c>0$ is a universal constant. 
     Let $\widetilde u$ be the solution to the Dirichlet problem
     $$
    \begin{cases}
        \partial_t \widetilde u - \Mm(D^2 \widetilde u)=0 & \text{ in } \widetilde Q,\\
        \widetilde u = \widetilde c & \text{ on } \widetilde T,\\
        \widetilde u = 0 & \text{ on } \p_p \widetilde Q.
    \end{cases}
     $$
     Then $u-\widetilde u \in \overline{\mathcal{S}}(0)$ in $\widetilde Q$ with $u\geq \widetilde u$ on $\p_p \widetilde Q$. By the maximum principle, $u\geq \widetilde u $ in $\widetilde Q$.
     By Lemma~\ref{lem:flathopf} (rescaled), there exist universal constants $\delta_1>0$ and $0<C_1<1/2$ such that
     $$
     u(x,t) \geq \widetilde u(x,t) \geq C_1 (x_n - c_0) \quad \text{ in } \C_{\delta_1}^+(c_0e_n,0).
     $$
     Since $u\geq 0$ in $\Omega\cap \C_1^+$, it follows that 
     $u(x,t) \geq C_1 (x_n-c_0)$ in $\Omega_{\delta_1}^+.$
     Taking $\widetilde a = C_1$, $a_0=0$, and $\rho \leq \delta_1$, and recalling that $A_0=c_0$, we get 
     $$
    \inf_{\Omega_{\rho}^+} ( u - \widetilde a x_n + a_0 x_n ) \geq - A_0,
     $$
     and thus, \eqref{eq:lowerest}--\eqref{eq:boundak} hold for $k=0$. 
     
     Assume the estimates hold for some $k\geq0$. Let $r=\rho^{k+1}$ and let $v$ solve
     $$
    \begin{cases}
        v_t - \Mm(D^2v) = 0 & \text{ in } \C_r^+,\\
        v=0 & \text{ on } \C_r',\\
        v = -\rho^k A_k & \text{ on } \p_p \C_r \setminus \C_r'.
    \end{cases}
     $$
         By the maximum principle, $-\rho^k A_k \leq v \leq 0$ in $\C_r^+$.
     Let $w = u - \widetilde a x_n + a_k x_n - v$. Then $w$ satisfies
     $$
    \begin{cases}
        w \in \underline{\mathcal{S}}(0) & \text{ in } \Omega_r^+,\\
        w\geq - \widetilde a x_n + a_k x_n & \text{ on } \p_p \Omega \cap \C_r^+,\\
        w \geq 0 & \text{ on } \p_p \C_r^+ \cap \overline\Omega,
    \end{cases}
     $$
     where the third inequality is satisfied by the induction hypothesis. 

     Next, we will estimate $v$ and $w$. By the boundary $C^{1,\alpha}$ estimate in Lemma~\ref{lem:flatbdryreg} and the maximum principle, there are universal $0<\bar \alpha<1$ and $\bar a\geq 0$ such that
     \begin{align*}
         \| v + \bar a x_n \|_{L^\infty(\C_{\rho r}^+)} & \leq C_1 \rho^{1+\bar\alpha} \|v\|_{L^\infty(\C_r^+)} \leq C_1 \rho^{1+\bar\alpha} \rho^kA_k 
         \leq C_1\rho^{\bar\alpha-\alpha_0} \cdot \rho^{k+1} A_{k+1},
     \end{align*}
      and
     $$
     \bar a \leq \frac{C_1}{r} \|v\|_{L^\infty(\C_r^+)}\leq \frac{C_1}{\rho^{k+1}} (\rho^k A_k)= C_1 \rho^{-1} A_k\leq C_1 \rho^{-1-\alpha_0} \cdot A_{k+1},
     $$
     where we used in both estimates that 
     $A_k \leq \rho^{-\alpha_0} A_{k+1}.$
     
     Take $\alpha_0=\bar \alpha/2$, $\rho < C_1 \rho^{\bar\alpha/2}$, and $\widetilde C = C_1 \rho^{-1-\alpha_0}$. Then
     \begin{equation} \label{eq:esta}
         \bar a \leq \widetilde C A_{k+1},
     \end{equation}
     and since $\Omega_{\rho^{k+2}}^+= \Omega_{\rho r}^+ \subseteq \C_{\rho r}^+$, we get
     \begin{align} \label{eq:estv}
         \|v +  \bar a x_n\|_{L^\infty(\Omega_{\rho^{k+2}}^+)} \leq \| v +  \bar a x_n\|_{L^\infty(\C_{\rho r}^+)}\leq C_1\rho^{\bar\alpha/2} \cdot \rho^{k+1} A_{k+1}\leq\frac{1}{2} \rho^{k+1} A_{k+1}. 
     \end{align}
To estimate $w$, we use the maximum principle, so that
\begin{equation} \label{eq:estw}
\inf_{\Omega_{\rho^{k+2}}^+}w \geq \inf_{\Omega_r^+} w \geq  \inf_{\p_p \Omega \cap \C_r^+} (- \widetilde a x_n + a_k x_n) \geq - \widetilde a r \omega(r) = -\widetilde a \rho^{k+1} \omega(\rho^{k+1}) \geq -\frac{1}{2} \rho^{k+1} A_{k+1},
\end{equation}
where we used that $A_{k+1} = \max\{\omega(\rho^k), \rho^{\alpha_0}A_{k-1} \} \geq \omega(\rho^k)\geq \omega(\rho^{k+1})$ and $0<\widetilde a < 1/2$.
Let $a_{k+1} = a_k + \bar a$. By \eqref{eq:esta}, the condition \eqref{eq:cauchy} holds for $k+1$.
Combining \eqref{eq:estv} and \eqref{eq:estw},
\begin{align*}
\inf_{\Omega_{\rho^{k+2}}^+} (u - \widetilde a x_n + a_{k+1} x_n )  &= \inf_{\Omega_{\rho^{k+2}}^+} (w+v+\bar a x_n) \\
&\geq \inf_{\Omega_{\rho^{k+2}}^+} w + \inf_{\Omega_{\rho^{k+2}}^+} (v+\bar ax_n)\\ 
&\geq -\frac{1}{2} \rho^{k+1} A_{k+1}+ -\frac{1}{2} \rho^{k+1} A_{k+1}\\
&= -\rho^{k+1} A_{k+1},
\end{align*}
so \eqref{eq:lowerest} is satisfied for $k+1$. It remains to check \eqref{eq:boundak}. Taking $c_0$ small enough,
\begin{align*}
    a_{k+1} \leq \sum_{j=0}^{k+1} |a_j-a_{j-1}| \leq \widetilde C \sum_{j=0}^\infty A_j \leq 4c_0 \widetilde C  \leq \frac{\widetilde a}{2}.
\end{align*}
By induction, the claim is true for all $k\geq 0$.
\end{customproof}

\section*{Acknowledgments}
This research was partially developed during D.J.'s visit to the Department of Mathematics at Rutgers University, whose warm hospitality is gratefully acknowledged. The visit was partially funded by University of Bologna funds \say{Attivit\`a di collaborazione con universit\`a del Nord America}, within the framework of the project \say{Interplaying problems in analysis and geometry}.

\Addresses
\end{document}